\documentclass[12pt]{article}
\usepackage[utf8]{inputenc}
\usepackage{amssymb, amsmath,amsthm, mathtools, mathrsfs}
\usepackage{bbm}
\usepackage{amsfonts}
\usepackage[left=2.5cm,right=2.5cm,top=2cm,bottom=2cm]{geometry}
\usepackage{microtype}

\usepackage{hyperref}
\usepackage{cleveref}

\usepackage{tabto}
\TabPositions{2 cm, 4 cm, 6 cm, 8 cm}

\usepackage{tikz-cd}

\usepackage{tikz}
\usetikzlibrary{arrows, automata, positioning}

\usepackage{tocloft}
\usepackage{appendix}

\usepackage{enumerate}

\usepackage{hyperref}
\hypersetup{
  colorlinks   = true, 
  urlcolor     = blue, 
  linkcolor    = blue, 
  citecolor   = blue 
}
 
\newtheorem{thmA}{Theorem}

\newtheorem{corA}[thmA]{Corollary}

\newtheorem{theorem}{Theorem}[section]
\newtheorem{corollary}[theorem]{Corollary}
\newtheorem{lemma}[theorem]{Lemma}
\newtheorem{claim}[theorem]{Claim}
\newtheorem{proposition}[theorem]{Proposition}

\theoremstyle{definition}
\newtheorem{definition}[theorem]{Definition}
\newtheorem{example}[theorem]{Example}
\newtheorem{remark}[theorem]{Remark}

\newtheorem{question}[theorem]{Question}

\newtheorem{convention}[theorem]{Convention}
\newtheorem*{acks}{Acknowledgements}
\newtheorem*{outl}{Outline}

\newtheorem*{lemma*}{Lemma}
\newtheorem*{proposition*}{Proposition}
\newtheorem*{theorem*}{Theorem}
\newtheorem*{corollary*}{Corollary}
\newtheorem*{remark*}{Remark}
\newtheorem*{remarks*}{Remarks}
\newtheorem*{convention*}{Convention}

\newcommand{\im}{\operatorname{im}}

\newcommand{\id}{\operatorname{id}}

\newcommand{\R}{\mathbb{R}}
\newcommand{\Z}{\mathbb{Z}}
\newcommand{\N}{\mathbb{N}}

\newcommand{\HH}{\operatorname{H}}

\newcommand{\ZZ}{\operatorname{Z}}

\newcommand{\hnb}{\mathrm{H}^n_{\mathrm{b}}}
\newcommand{\hib}{\mathrm{H}^i_{\mathrm{b}}}
\newcommand{\hsb}{\mathrm{H}^*_{\mathrm{b}}}
\newcommand{\htb}{\mathrm{H}^2_{\mathrm{b}}}
\newcommand{\hbb}{\mathrm{H}^\bullet_{\mathrm{b}}}

\newcommand{\diam}{\mathrm{diam}}
\newcommand{\homeo}{\mathrm{Homeo}}
\newcommand{\homeoo}{\mathrm{Homeo}_{\circ}}
\newcommand{\homeop}{\mathrm{Homeo}_{+}}
\newcommand{\homeoc}{\mathrm{Homeo}_{\mathrm{c}}}
\newcommand{\homeoco}{\mathrm{Homeo}_{\mathrm{c}, \circ}}
\newcommand{\diffr}{\mathrm{Diff}^r}
\newcommand{\diffro}{\mathrm{Diff}^r_{\circ}}
\newcommand{\diffrp}{\mathrm{Diff}^r_{+}}
\newcommand{\diffrc}{\mathrm{Diff}^r_{\mathrm{c}}}

\newcommand{\diff}{\mathrm{Diff}}

\newcommand\PLp{\mathrm{PL}_{+}}
\newcommand\PLpw{\mathrm{PL}_{+}^{\omega}}
\newcommand\PPpw{\mathrm{PP}_{+}^{\omega}}

\newcommand{\lra}{\longrightarrow}
\newcommand{\Ctilde}{\widetilde{C}}
\newcommand{\PO}{\mathcal{P}}
\newcommand{\fat}{\mathcal{F}}
\newcommand{\X}{\mathcal{X}}
\newcommand{\emb}{\operatorname{Emb}_{<}}

\begin{document}

\title{The bounded cohomology of transformation groups \\ of Euclidean spaces and discs

\medskip

La cohomologie born{\'e}e des groupes de transformations \\ des espaces euclidiens et des disques}
\author{Francesco Fournier-Facio, Nicolas Monod and Sam Nariman \\ Appendix by Alexander Kupers}

\date{\today}
\maketitle

\begin{abstract}
We prove that the groups of orientation-preserving homeomorphisms and diffeomorphisms of $\R^n$ are boundedly acyclic, in all regularities. This is the first full computation of the bounded cohomology of a transformation group that is not compactly supported, and it implies that many characteristic classes of flat $\R^n$-bundles and $S^n$-bundles are unbounded. We obtain the same result for the group of homeomorphisms of the closed disc that restrict to the identity on the boundary, and for the homeomorphism group of the non-compact Cantor set. In the appendix, Alexander Kupers proves a controlled version of the annulus theorem which we use to study the bounded cohomology of the homeomorphism group of the discs.
\end{abstract}

\renewcommand{\abstractname}{R{\'e}sum{\'e}}

\begin{abstract}
Nous établissons l'annulation de la cohomologie bornée des groupes d'homéomorphismes et de difféomorphismes de $\R^n$, en toutes classes de régularité. Cela constitue la première détermination complète de la cohomologie bornée de groupes de transformations en dehors des cas compactement supportés; il en suit que diverses classes caractéristiques de $\R^n$-fibrés et $S^n$-fibrés sont non bornées. Nous obtenons le résultat correspondant pour les homéomorphismes du disque fermé qui sont triviaux sur son bord et pour les homéomorphismes de l'espace de Cantor non compact. Dans un appendice, Alexander Kupers démontre une version contrôlée du théorème de l’anneau, laquelle nous appliquons à l'étude du cas du disque.
\end{abstract}

\section{Introduction}
\label{sec:intro}
\subsection{Euclidean spaces vs. compactness}
The topological group $\homeo(\R^n)$ of all self-homeomorphisms of Euclidean space remains very mysterious to this day. Its homotopy type is not completely understood and neither is its group cohomology, despite many deep discoveries, especially in recent times (see \cite{MR4680344}, \cite{galatius2022algebraic} and the references therein). In particular, the results obtained so far indicate highly non-trivial higher homotopy groups.

By deep results of Thurston~\cite{thurston1974foliations} and McDuff~\cite{mcduff1980homology}, when $M$ is a connected, topological $n$-manifold, possibly with boundary and with trivial ends (such $M$ is contained in a compact manifold $\overline{M}$, where $\overline{M}\backslash M$ is a possibly empty union of components of the boundary of $\overline{M}$), the cohomology of the classifying space of the topological group $\homeo(M)$ coincides with the group cohomology of the underlying ``abstract'' (discrete) group. McDuff~\cite[Proposition 2.1]{mcduff1980homology} also proved the relative case, for example, when $M$ is a compact manifold with a non-empty boundary, the same statement holds for $\homeo(M, \text{rel }\partial M)$, which is the subgroup of $\homeo(M)$ consisting of homeomorphisms whose support lies in the interior of $M$.

In contrast to the mystery surrounding the group $\homeo(\R^n)$, Mather proved in 1971~\cite{mather} that the group $\homeoc(\R^n)$ of \emph{compactly supported} homeomorphisms of $\R^n$, as an abstract group, is $\Z$-acyclic. Moreover, when $M$ is a manifold for which Thurston--McDuff's previously mentioned result holds, the group cohomology of $\homeo(M)$ is completely determined if we know the full homotopy type of $\homeo(M)$ as a topological group. For example, $\homeoc(\R^n)$ as a topological group is contractible by the Alexander trick, and Mather's result proves its acyclicity as a discrete group, which is a base case for Thurston--McDuff's theorem. Moreover, in dimensions $2$ and $3$, the homotopy type of $\homeo_0(M)$, the identity component of $\homeo(M)$, as a topological group is known for ``geometric'' manifolds (see \cite{EE1} for the case of surfaces and \cite[Section 4.7]{BamlerICM2022} for geometric $3$-manifolds).

\medskip

Our contribution in this article is to the \emph{bounded cohomology} $\hbb$ of homeomorphism and diffeomorphism groups (always as abstract groups). In the compactly supported case, the analogue of Mather's theorem was established in 1985 by Matsumoto--Morita~\cite{matsumor}: $\homeoc(\R^n)$ is \emph{boundedly acyclic}. More recently~\cite{MN}, the bounded cohomology of $\homeo(M)$ and $\diff(M)$ has been determined for the compact case of $M=S^1$.

\medskip

Our main result is the first complete determination of the bounded cohomology of transformation groups without any compactness assumption, answering \cite[Question 7.6]{MN}.

\begin{thmA}[Theorems \ref{thm:criterion:R} and \ref{thm:criterion:Rn}]
\label{thm:main:homeo}
    For all $n \in \N^*$, the group $\homeop(\R^n)$ is boundedly acyclic, that is, $\hib(\homeop(\R^n)) = 0$ for all $i \geq 1$.
\end{thmA}

Before moving on, let us make clear the convention that we use throughout the article regarding topologies and classifying spaces of transformation groups.

\begin{convention*}
    If $G$ is a topological group, $BG$ denotes its classifying space as a topological group, and $BG^\delta$ its classifying space as a discrete group. When referring to the [bounded] cohomology of $G$, we always see $G$ as a discrete group, so this corresponds to the [bounded] cohomology of $BG^\delta$.
\end{convention*}

Some results in bounded cohomology are analogues of ordinary cohomological statements, e.g. Matsumoto--Morita's extension of Mather's result. By contrast, \Cref{thm:main:homeo} is genuinely different: even though the ordinary cohomology is not yet completely known, it is known to have a remarkably rich structure when the dimension $n$ is high enough \cite[Section 3.3]{MR4680344}.

Moreover, it is known that the group cohomology of transformations of the Euclidean space depends on regularity. For a manifold $M$ as before, it is a consequence of a remarkable theorem of Tsuboi \cite{tsuboi1989foliated} that the cohomology of $B\diff^1(M)$, the classifying space of the group $\diff^1(M)$ with $C^1$-topology, is the same as the group cohomology of $\diff^1(M)$ as a discrete group. However, in regularities $r>1$, it is known that the group cohomology of $\diff^r(M)$ is different from the cohomology of $B\diff^r(M)$. In particular, when $M=\R^n$, the topological group $\diff^r(\R^n)$ for $r>0$ is homotopy equivalent to $O(n)$ but as a discrete group, its group cohomology is the same as the cohomology of $B\Gamma_n^r$, the classifying space of $C^r$-Haefliger structures of codimension $n$ \cite[Proposition 3.1 and Proposition 1.3]{segal1978classifying}. The cohomology of $B\Gamma_n^r$ is related to the secondary characteristic classes of foliations of codimension $n$ and it is known to be highly non-trivial (see Section \ref{sec:pontryagin} and \cite{MR0928396}). Still, we also obtain the analogue of Theorem \ref{thm:main:homeo} for diffeomorphisms in all regularities.

\begin{thmA}[Theorems \ref{thm:criterion:R} and \ref{thm:criterion:Rn}]
\label{thm:main:diff}
    For all $r \in \N^* \cup \{ \infty \}$ and all $n \in \N^*$, the group $\diffrp(\R^n)$ is boundedly acyclic.
\end{thmA}

\begin{remarks*}\label{rem:1}
  \begin{enumerate}[(i)]
\item The above formulation in terms of \emph{orientation-preserving} homeomorphism and diffeomorphism groups $\homeop{}, \diffrp{}$ is more natural for our proofs, but it immediately implies the result for the entire groups (see~\cite[Corollary 8.8.5]{monod}). It is standard that for $r>0$, the derivative at the origin induces a homotopy equivalence between $\diffrp(\R^n)$ and $\mathrm{GL}_n(\R)_+$ as topological groups. So we have $\diffrp(\R^n)=\diffro(\R^n)$ for $r>0$, where $\diffro$ denotes the identity component. The same statement holds for $r=0$ by the deep theorem in \cite[Corollary of Theorem 2]{Kirby}. We will use the corresponding notation $\homeoo, \diffro$ when this point of view is more natural.
\item Ordinary acyclicity usually refers to the vanishing of homology with integral coefficients, which then implies the vanishing of cohomology with arbitrary trivial coefficients. \emph{Bounded acyclicity} is defined directly as the vanishing of bounded cohomology with real coefficients, because fundamental techniques of the theory are unavailable over the integers (thus we will always speak of \emph{$\Z$-acyclicity}, to avoid confusion). While such techniques are not explicitly used in this paper, they are essential in bounded acyclicity results for compactly supported transformation groups \cite{lamplighters, ccc}, which we use as input. Nevertheless, bounded acyclicity can be understood as a version of homological acyclicity but with a control of the norm in the spirit of homological isoperimetry, see Matsumoto--Morita~\cite{matsumor}.
  \end{enumerate}
\end{remarks*}

As suggested by the fact that our results differ from the picture in ordinary cohomology, our proof will require a new device. That device can also be used in ordinary cohomology; as an illustration, the special case $n=1$ of Theorem \ref{thm:main:homeo} actually goes through in ordinary cohomology and recovers the following result of McDuff (which she also attributes to Thurston and Segal). Contrary to the original proof, this does not rely on the topology of the group $\homeop(\R)$, nor on foliation theory and Haefliger spaces.

\begin{thmA}[Theorem \ref{thm:criterion:R}, \cite{mcduff1980homology}]
\label{thm:acyclic}
    The group $\homeop(\R)$ is $\Z$-acyclic.
\end{thmA}

A theorem of Ghys \cite{ghys} states that the bounded Euler class in $\htb(\homeop(S^1); \Z)$ classifies actions on the circle by orientation-preserving homeomorphisms up to semiconjugacy. Ghys's Theorem can be used to prove that, given a countable discrete group $\Gamma$, the semiconjugacy relation on the representation variety $\mathrm{Hom}_{\mathrm{irr}}(\Gamma, \homeop(S^1))$ is \emph{smooth}. Here $\mathrm{Hom}_{\mathrm{irr}}$ denotes those representations without global fixed points, and smooth equivalence relations are the simplest ones from the point of view of Borel reducibility. On the other hand, the semiconjugacy relation on the representation variety $\mathrm{Hom}_{\mathrm{irr}}(\Gamma, \homeop(\R))$ is \emph{essentially hyperfinite}, a strictly more complex type of equivalence relation, already when $\Gamma$ is a free group \cite[Section 3.5]{locallymoving}. This is in agreement with the fact that $\homeop(\R)$ has vanishing bounded cohomology with integer coefficients in all positive degrees, which is a direct consequence of Theorems \ref{thm:main:homeo} and \ref{thm:acyclic} (\Cref{cor:boundedly:Z:acyclic}).

\subsection{Taming infinity: a sketch}
The method that we introduce to study the various transformation groups of $\R^n$ is a general paradigm applicable in many other situations, some of which will be mentioned later in this article. We now first sketch it in one of the simplest possible cases, namely for the group $\homeop(\R)$ of (orientation-preserving) homeomorphisms of the line (Section \ref{verifying:line}).

\medskip
The most naive attempt to reduce $\homeop(\R)$ to the compactly supported case is to cut $\R$ into an infinite sequence of compact intervals. Since we prefer usual sequences indexed by $\N$ rather than bi-infinite sequences, we shall work with the subgroup $G$ of elements that are trivial in a neighborhood of $-\infty$ and we consider a sequence of points that increases to $+\infty$. Once we prove that this group is boundedly acyclic, an elementary symmetry argument will allow us to take care of the germs at $-\infty$ in a similar way to conclude that $\homeop(\R)$ itself is boundedly acyclic.

To be more precise, we consider increasing diverging sequences $x=(x_n)$ of \emph{fat points} $x_n$. As in~\cite{MN} a fat point is a germ at $0$ of embeddings of an interval $(-1,1)$ into $\R$, its \emph{core} $\mathring x_n$ is the point $\mathring x_n = x_n(0)\in \R$ (Definition \ref{def:R:fatpoint}). The group $G$ acts on fat points by post-composition, so a homeomorphism fixing $x_n$ must be trivial in a neighborhood of the core. In conclusion, the stabilizer in $G$ of the entire sequence of fat points is the product group
\[ \homeoc\big((-\infty, \mathring x_1)\big) \times \prod_{j=1}^\infty \homeoc\big((\mathring x_j, \mathring x_{j+1})\big). \]
That stabilizer, a power of compactly supported groups, is known to be [boundedly] acyclic by previous work (Lemma \ref{lem:R:fatsequence:stabilizers}). Therefore a spectral sequence argument, as in \cite{MN}, reduces the problem to the study of two semisimplicial sets: on the one hand, the collection of all fat sequences $x$ as above; on the other hand, the quotient of this semisimplicial object by the $G$-action (Theorem \ref{thm:resolutions}). An important aspect of our approach is that the semisimplicial structure is given by the \emph{subsequence partial order} on spaces of sequences. We shall see that, in great generality, this is a $\Z$-acyclic and boundedly acyclic object (Corollary \ref{cor:W}). 

\medskip
In any case, so far all this construction has merely kicked the difficulty further down the road: the ``mystery at infinity'' is now contained in the abstract semisimplicial set defined by $G$-orbits of increasing sequences under the order relation of taking subsequences.

At that point we observe that homeomorphisms of $\R$ act transitively on sequences of fat points that increase to infinity (Lemma \ref{lem:R:transitivity}) and that, in fact, the entire structure of the quotient semisimplicial object is determined by the abstract order relation on \emph{indexing} the subsequences. In precise terms, we show that this quotient semisimplicial set can be identified with the nerve of the monoid  $\emb(\N)$ of order-preserving embeddings $\N\to\N$ (Proposition \ref{prop:orbitcomplex}). The remaining part of the proof will therefore be to establish:

\begin{thmA}[Theorems \ref{thm:emb:bac} and \ref{thm:emb:acyclic}]
\label{thm:emb}
The monoid $\emb(\N)$ is $\Z$-acyclic and boundedly acyclic.
\end{thmA}

It should appear plausible now that the above sketch can be implemented in many other situations. For instance, in the case of $\R^n$, we shall consider suitably concentric sequences of fat spheres so that the stabilizer becomes an infinite product of groups of compactly supported homeomorphisms of annuli, and one ball (Section \ref{verifying:Rn}). Each aspect of the above outline will raise new difficulties in various situations, except the final step: the quotient structure will always be the same monoid $\emb(\N)$.

\medskip
It remains therefore to explain how we prove \Cref{thm:emb}, which encodes the common residual difficulty at infinity of the diverse situations. The above spectral sequence argument is entirely reversible: if we exhibit \emph{some} [boundedly] acyclic group $G$ acting on some [boundedly] acyclic space of sequences with [boundedly] acyclic stabilizers and the same quotient object $\emb(\N)$, then the [bounded] acyclicity of $\emb(\N)$ will follow.

This leaves us with the freedom to choose the most favorable example. One possibility is the group of all countably supported permutations of an uncountable set (Section \ref{ss:auxiliary}). This group is \emph{dissipated}, a property well-known to imply $\Z$-acyclicity via the concept of binate (or pseudo-mitotic) group~\cite{varadarajan, berrick} (\Cref{def:binate}). This property also implies bounded acyclicity~\cite{binate} and therefore we will obtain a proof of \Cref{thm:emb}. In closing, we observe that the choice of this group is not unnatural: there are uncountably many ways to pass to subsequences, but each one is of course a rearrangement of countably many items.

\subsection{Beyond Euclidean spaces}

As mentioned above, our proof of Theorems \ref{thm:main:homeo} and \ref{thm:main:diff} relies on a general criterion for bounded acyclicity. This applies to a variety of groups, which were not approachable by the previous methods tailored to compact spaces. 

\medskip

An instance of this is in the zero-dimensional setting of Cantor sets. The $\Z$-acyclicity of the homeomorphism group of the Cantor set was shown by Tsuboi and Sergiescu \cite{ts:cantor} (and later generalized to the more general Menger spaces \cite{ts:menger}). A similar strategy also proves bounded acyclicity, as shown by Andritsch \cite{andritsch}. Our methods allow us to treat the non-compact Cantor set, which is uniquely defined as a locally compact, non-compact, metrizable, totally disconnected space without isolated points. To the best of our knowledge, also the $\Z$-acyclicity result is new.

\begin{thmA}[Theorem \ref{thm:cantor:criterion}]
\label{thm:cantor}
    The homeomorphism group of the non-compact Cantor set is $\Z$-acyclic and boundedly acyclic.
\end{thmA}

In another direction, we consider homeomorphism groups of discs fixing the boundary pointwise. While discs are compact, from the point of view of homology and bounded cohomology, their transformation groups are closer to the groups of Theorem \ref{thm:main:homeo}: the difficulties present at infinity are now present at the boundary. In this case, the bounded acyclicity was already shown (in all regularities) for $D^2$ \cite[Theorems 1.3 and 1.4]{MN}, however the proof is dependent on certain combinatorial properties of (fat) chords in the $2$-disc that fail in higher dimensions.

\begin{thmA}[Theorem \ref{thm:criterion:discs}]
\label{thm:discs}
    For all $n \in \N^*$, the group $\homeo(D^n, \partial)$ of homeomorphisms of the closed disc $D^n$ fixing the boundary pointwise, is boundedly acyclic.
\end{thmA}

We obtain the following consequence, which generalizes \cite[Theorem 1.3]{MN} and answers \cite[Question 7.5]{MN}:

\begin{corA}
\label{cor:discs}
    For all $n \in \N^*$, the restriction $\homeop(D^n) \to \homeop(S^{n-1})$ induces an isomorphism in bounded cohomology in all degrees.
\end{corA}

This allows to determine the bounded cohomology of $\homeop(D^n)$ in low degrees (Corollary \ref{cor:disc:low}).

These are just two more examples of the applicability of our methods, and we do not aim for a complete list. In Section \ref{sec:further} we discuss further instances in which our methods apply, and speculate about other settings where deploying our techniques seems to require additional work.

Other applications follow from our main results by some standard manipulations. A notable instance is the case of compact annuli, where we prove:

\begin{thmA}
\label{thm:annuli:homeo}
    For all $n \in \N^*$, the identity component $\homeoo(S^n \times [0, 1], \partial)$ of the group of homeomorphisms of the annulus $S^n \times [0, 1]$ fixing both boundary components pointwise, is boundedly acyclic.
\end{thmA}

As in the case of the disc, without imposing that the boundary is fixed we obtain a result on restrictions (Corollary \ref{cor:annulus:restriction}) and computations in low degree (Corollary \ref{cor:annulus:low}). Leveraging results from \cite{MN}, we can treat all regularities in dimension $2$:

\begin{corA}
\label{cor:annuli:regularity}
    For all $r \in \N$, the restriction $\diffro(S^1 \times [0, 1]) \to \diffro(S^1) \times \diffro(S^1)$ induces an isomorphism in bounded cohomology in all degrees. Explicitly, denoting by $\mathscr{E}_i \in \htb(\diffro(S^1 \times [0, 1]))$ the bounded Euler class for the action of $\diffro(S^1 \times [0, 1])$ on the boundary component $S^1 \times \{ i \}$, we have an isomorphism of graded $\R$-algebras:
    \[\hsb(\diffro(S^1 \times [0, 1])) \cong \R[\mathscr{E}_0, \mathscr{E}_1].\]
\end{corA}

In degree $2$, we recover a theorem of Militon (see Definition \ref{def:torsionnumber} for the torsion number).

\begin{corA}[\cite{militon}]
\label{cor:militon}
    For all $r \in \N$ and $G = \diffro(S^1 \times [0, 1])$, the space $Q(G)/\HH^1(G)$ of non-trivial quasimorphisms is one-dimensional, spanned by the torsion number. In particular, for $r \neq 2, 3$, the space $Q(G)$ of homogeneous quasimorphisms is one-dimensional, spanned by the torsion number.
\end{corA}

In the last statement in \Cref{cor:militon}, the reason that $r=2,3$ are exceptions is that for these regularities, the perfectness of $G = \diffro(S^1 \times [0, 1])$ is not known.

\subsection{Unboundedness of characteristic classes}

The bounded acyclicity of $\homeop(\R^n)$ implies that the Milnor--Wood inequality fails for certain characteristic classes of flat $\R^n$-bundles and flat $S^n$-bundles. To describe these classes, first recall again the remarkable theorem of Thurston  (\cite[Corollary (b) of Theorem 5]{thurston1974foliations} and see \cite[Section 2, Theorem 2.5]{mcduff1980homology} for the proof) for homeomorphisms: let $M$ be a connected, topological
$n$-manifold, possibly with boundary and with trivial ends. We have the natural map between classifying spaces
\begin{equation}
B\homeo(M)^\delta\to B\homeo(M),
\end{equation}
where $\homeo(M)^\delta$ is the group of homeomorphisms of $M$ with the discrete topology. This map is acyclic and in particular, it induces a homology isomorphism in all degrees. The same statement also holds in the relative case e.g. for manifolds with boundary the same statement holds for $\homeo(M, \text{rel }\partial M)$ which is the subgroup of $\homeo(M)$ consisting of homeomorphisms whose supports are away from the boundary \cite[Section 2, Theorem 2.5]{mcduff1980homology}. 

Therefore, the map
\begin{equation}\label{Thurston}
B\homeo(\R^n)^\delta\to B\homeo(\R^n),
\end{equation}
induces a homology isomorphism in all degrees. The same also holds for $\homeop(\R^n)$. There have been recent breakthroughs in understanding the homotopy type and the rational cohomology of $B\homeo(\R^n)$. In particular, there are topological Pontryagin classes $p_i$ for Euclidean bundles that are defined as rational cohomology classes, and for oriented even dimensional Euclidean bundles, there is also an Euler class $e$. Galatius and Randal-Williams \cite{galatius2022algebraic} proved a remarkable result which implies that the map
 \[
 \R[e,p_1, p_2, \dots]\to \HH^*(B\homeop(\R^{2n});\R),
 \]
 is injective for $2n\geq 6$, and in the odd-dimensional case, it follows from \cite[Corollary 1.2]{galatius2022algebraic} that
  \[
 \R[p_1, p_2, \dots]\to \HH^*(B\homeop(\R^{2n+1});\R),
 \]
 is injective for $2n+1\geq 7$. In low dimensions, we also know that $\homeop(\R^3)\simeq \mathrm{SO}(3)$ and $\homeop(\R^2)\simeq \mathrm{SO}(2)$ (\cite[Theorem 1.2.3]{MR0334262} and \cite{Balcerak1981HomotopyTO}). 
 Therefore $\R[p_1]\to \HH^*(B\homeop(\R^3);\R)$ and $\R[e]\to \HH^*(B\homeop(\R^2);\R)$ are isomorphisms. 
 The homology isomorphism in (\ref{Thurston}) implies that the same holds for $\HH^*(B\homeop(\R^n)^\delta;\R)$ for the corresponding $n$.

 If a non-trivial class in $\HH^*(B\homeo(M)^\delta;\R)$ is not in the image of the comparison map
 \[
 \hsb(B\homeo(M)^\delta;\R)\to \HH^*(B\homeo(M)^\delta;\R),
 \]
 we say that class is \emph{unbounded}: this can be interpreted geometrically as the failure of the Milnor--Wood inequality meaning that there exist a sequence of flat $M$-bundles over possibly varying base manifolds whose corresponding characteristic invariants divided by the simplicial volume of their bases go to infinity (see Proposition \ref{prop:unboundedness} for the precise statement). 
 
 A corollary of our main theorem is that any non-trivial class in $\HH^*(B\homeop(\R^n)^\delta;\R)$ is unbounded. More specifically we have the following.
 
 \begin{corA}\label{cor:flat R^n bundles}
The following non-trivial classes on homeomorphisms of the Euclidean spaces as discrete groups have no bounded representatives: 
\begin{itemize}
    \item For $2n\geq 6$, the image of the map $\R[e,p_1, p_2, \dots]\hookrightarrow \HH^*(B\homeop(\R^{2n})^\delta;\R)$,
    \item For $n=2$, the image of $\R[e]\hookrightarrow \HH^*(B\homeop(\R^{2})^\delta;\R)$,
    \item For $2n+1\geq 7$, the image of $\R[p_1, p_2, \dots]\hookrightarrow \HH^*(B\homeo(\R^{2n+1})^\delta;\R)$,
    \item For $n=3$, the image of $\R[p_1]\hookrightarrow \HH^*(B\homeo(\R^{3})^\delta;\R)$.
\end{itemize} 
 \end{corA}
 
Similarly, for $C^r$-flat $\R^n$-bundles when $r>0$, every non-trivial class in $\HH^*(B\diffrp(\R^n)^\delta;\R)$ is unbounded. To detect non-trivial classes in $\HH^*(B\diffrp(\R^n)^\delta;\R)$, we shall use a deep theorem of Segal \cite[Prop.~1.3 and~3.1]{segal1978classifying}, that there exists a map $$B\diffrp(\R^n)^\delta\to B\mathrm{S}\Gamma^{r}_{n},$$ which is a homology isomorphism, where $B\mathrm{S}\Gamma^{r}_{n}$ is the classifying space of Haefliger structures for codimension $n$ foliations that are transversely oriented. There is extensive literature on secondary characteristic classes of foliations that give non-trivial classes in $\HH^*(B\mathrm{S}\Gamma^{r}_{n}; \R)$ when $r>1$. In particular, there is a Godbillon--Vey class in $\HH^{2n+1}(B\mathrm{S}\Gamma^{r}_{n}; \R)$ that is non-trivial (see \cite{MR0928396} for the corresponding cocycle formula in $\HH^{2n+1}(B\diffrp(\R^n)^\delta;\R)$). To detect the non-triviality of classical characteristic classes, we consider a map 
\[
\nu\colon B\mathrm{S}\Gamma^{r}_{n}\to B\mathrm{GL}_n(\R)_{+},
\]
which classifies oriented normal bundles to the codimension $n$ foliations where $\mathrm{GL}_n(\R)_{+}$ is the group of invertible matrices with positive determinant. For all regularities, it is known that the map $\nu$ is at least $(n+1)$-connected, see~\cite[Remark~1, Section~II.6]{haefliger1971homotopy}. Therefore, the map
\[
\nu^*\colon \HH^*(B\mathrm{GL}_n(\R)_{+};\R)\to \HH^*(B\mathrm{S}\Gamma^{r}_{n};\R),
\]
is injective in degrees $*\leq n+1$ for all regularities. For $r=1$, it is an isomorphism in all degrees by a remarkable theorem of Tsuboi \cite{tsuboi1989foliated} and for $r>1$, it is conjecturally injective in degrees below $2n$ \cite[Problem 14.5]{hurder2003foliation}. So we obtain the following corollary about the unboundedness of the Euler class and Pontryagin classes.
\begin{corA}[Theorem \ref{C^r unboundedness}]
    The polynomials of degree less than $n+2$ on Pontryagin classes and the Euler class (when $n$ is even) in $\HH^{*}(B\diffrp(\R^n)^\delta;\R)$ for $r>1$ have no bounded representatives. When $r=1$, the same statement holds without any condition on the degree of the polynomials.
\end{corA}

We remark that Calegari in \cite{Calegari} proved that the Euler class is unbounded for oriented $C^r$-flat $\R^2$-bundles.

\medskip

These unboundedness results for classical characteristic classes of flat $\R^n$-bundles are in contrast to Milnor, Sullivan, Smillie, and Gromov's results on the boundedness of these classes for flat \emph{linear} $\R^n$-bundles. For example, Milnor showed that the Euler class is bounded for flat linear $\R^2$-bundles \cite{MR95518}; Sullivan~\cite{MR0418119} and  Smillie~\cite{Smillie} generalized it to flat linear $\R^{2n}$-bundles. Gromov further generalized it \cite[Page 23]{MR0686042} by proving that if $G$ is an algebraic subgroup in the linear group $\mathrm{GL}_n(\R)$, then any class in the image of the map
\[
\HH^*(BG; \R)\to \HH^*(BG^{\delta}; \R),
\]
is a bounded class.

For flat $S^n$-bundles, we also have topological Pontryagin classes and the Euler class (if it is oriented and  $n$ is odd). These classes are induced by the following composition
 \[
 B\homeo(S^{n})^\delta\to B\homeo(D^{n+1})^\delta\to B\homeo (\R^{n+1})^\delta,
 \]
 where the first map is induced by coning the sphere to a disc and the second map is induced by the restriction to the interior of the disc. So we can pull back the classes in $\HH^*(B\homeoo (\R^{n+1})^\delta; \R)$ to $\HH^*(B\homeoo(S^{n})^\delta; \R)$. As a consequence of Galatius--Randal-Williams's result \cite[Theorem 1.4]{galatius2022algebraic}, we shall prove in Section \ref{sec:pontryagin} that
 \[
 \R[p_1, p_2, \dots]\to \HH^*(B\homeo(S^{n})^\delta;\R), \]
 is injective for $n\geq 6$. We then use the bounded acyclicity of $\homeo(\R^n)$ to obtain the following unboundedness result for the invariants of flat sphere bundles.
 \begin{thmA}\label{thm:spherebundles}
     The non-trivial classes in the image of $\R[p_1, p_2, \dots]$ in $\HH^*(B\homeo(S^{n})^\delta;\R)$ have no bounded representatives for $n\geq 7$ and for $n=4,5$ and $6$ the same statement holds for classes in the image of $\R[p_1, p_2]$. For $n=2, 3$, the same statement holds for classes in the image of $\R[p_1]$.
 \end{thmA}

The unboundedness of $p_1$ in $\HH^4(B\homeo(S^3)^\delta;\R)$ was first shown in \cite[Theorem 1.8]{MN}, answering a question of Ghys \cite[F.1]{langevin}. Theorem \ref{thm:spherebundles} partially answers \cite[Questions 7.3 and 7.4]{MN}.

\begin{outl}
    We begin by stating the criteria for bounded acyclicity and $\Z$-acyclicity in Section \ref{sec:criterion}. In Section \ref{sec:verifying} we go through the list of groups from the introduction, and verify that they satisfy the criteria. The case of the disc presents some additional technical difficulties which are solved by a controlled annulus theorem, proven in Appendix \ref{appendix}. We prove the criteria in Section \ref{sec:proof}. Section \ref{sec:pontryagin} is dedicated to unboundedness results for characteristic classes of flat $\R^n$-bundles and $S^n$-bundles. Finally in Section \ref{sec:further}, we present further computations (including the results on compact annuli), and some speculations.
\end{outl}

\begin{acks}
    FFF was supported by the Herchel Smith Postdoctoral Fellowship Fund. SN  was partially supported by  NSF CAREER Grant DMS-2239106 and Simons Foundation Collaboration Grant (855209). The authors are indebted to Benjamin Br\"uck, S{\o}ren Galatius, Manuel Krannich, Nicol\'as Matte Bon, Samuel Mu{\~n}oz-Ech{\'a}niz, Oscar Randal-Williams and Shmuel Weinberger for useful conversations, and to Shigeyuki Morita for insightful comments. They also thank the anonymous referees for their careful reading and helpful suggestions. Finally, they thank Alexander Kupers for writing the proof of the controlled annulus theorem in the appendix, which is a crucial ingredient in the proof of Theorem \ref{thm:discs}.
\end{acks}

\section{A criterion for bounded acyclicity}
\label{sec:criterion}

We start by formulating a general criterion for bounded acyclicity of groups admitting a certain action on a poset. This criterion will be the common denominator of all our proofs.

\begin{definition}\label{def:W}
    Let $\PO$ be a poset. We say that $\PO$ satisfies the $\mathbf{W}$ property, if for every finite subposet $\mathcal{Q}$ with minimal elements $\{x_1, \ldots, x_k\}$ the following holds. For every subset $I \subseteq \{1, \ldots, k\}$ there exists $y_I \in \PO$ such that
    
    \begin{enumerate}
        \item if $I \subseteq J$ then $y_I \preceq y_J$;
        \item for $x\in\mathcal{Q}$, if $x_i \preceq x$ for all $i \in I$ then $y_I \preceq x$.
    \end{enumerate}
\end{definition}

Let $G$ be a group acting on a set $X$. The set of all sequences $\mathbf{x} = (x_i)_{i \in \N^*}, x_i \in X$ of pairwise distinct elements $x_i$ has a natural poset structure, where $\mathbf{x} \preceq \mathbf{y}$ if $\mathbf{x}$ is a subsequence of $\mathbf{y}$. The action of $G$ on $X$ induces an order-preserving action on this poset of sequences. Any $G$-invariant subposet $\X$ will be called a \emph{$G$-poset of sequences in $X$}.

The $\mathbf{W}$-property occurs naturally among posets of sequences. The following is a motivating example, we will see a more elaborate version in Section \ref{verifying:Rn}.

\begin{example}
    Let $X$ be the set of Euclidean circles in $\R^2$. Let $\X$ be the set of sequences $\mathbf{x} = (x_i)_{i \in \N^*}$ that are concentric and going to infinity, i.e. such that $x_i$ is contained in $x_{i+1}$ for all $i \in \N^*$, and such that every compact subset of $\R^2$ is contained in $x_j$ for some $j \in \N^*$. Then $\X$ is a poset with inclusion and has the $\mathbf{W}$ property. This will follow for the arguments in Section \ref{verifying:Rn}, however the following picture verifies the condition for a subposet $\mathcal{Q} = \{x_1, x_2\}$.

    \bigskip

\begin{tikzpicture}[line cap=round, line join=round]

\definecolor{myblue}{RGB}{65,105,225}
\definecolor{myred}{RGB}{220,20,60}


\draw[myblue, dashed, opacity=0.95, line width=1.25pt] (0,0) circle (3.16);
\draw[myblue, opacity=0.95, line width=1.25pt]        (0,0) circle (2.13);
\draw[myblue, dashed, opacity=0.95, line width=1.25pt] (0,0) circle (0.90);
\draw[myblue, opacity=0.95, line width=1.25pt]        (0,0) circle (0.30);

\draw[myred, opacity=0.95, line width=1.25pt]         (0.80,0) circle (3.06);
\draw[myred, dashed, opacity=0.95, line width=1.25pt]  (0.80,0) circle (1.83);
\draw[myred, opacity=0.95, line width=1.25pt]         (0.80,0) circle (1.16);
\draw[myred, dashed, opacity=0.95, line width=1.25pt]  (0.80,0) circle (0.23);

\begin{scope}[xshift=8.3cm]
  \draw[myred,  line width=1.35pt] (0.80,0) circle (3.06);
  \draw[myblue, line width=1.35pt] (0,0)    circle (2.13);
  \draw[myred,  line width=1.35pt] (0.80,0) circle (1.16);
  \draw[myblue, line width=1.35pt] (0,0)    circle (0.30);
\end{scope}

\end{tikzpicture}

\bigskip

Here $x_1$ (respectively, $x_2$) is the sequence of full and dashed blue (respectively, red) circles in the left picture. We let $y_1$ (respectively $y_2$) be the sequence of full blue (respectively, red) circles. The right picture is $y_{1, 2}$. These verify the condition of the $\mathbf{W}$ property, as seen in the diagram below, which justifies the name.

\[\begin{tikzcd}
	{x_1} &&&& {x_2} \\
	&& {y_{\{1,2\}}} \\
	& {y_1} && {y_2}
	\arrow[no head, from=1-1, to=3-2]
	\arrow[no head, from=3-2, to=2-3]
	\arrow[no head, from=2-3, to=3-4]
	\arrow[no head, from=3-4, to=1-5]
\end{tikzcd}\]
\end{example}

\begin{theorem}
\label{thm:criterion}
	Let $G$ be a group acting on a set $X$ and let $\X$ be a non-empty $G$-poset of sequences in $X$. Suppose that the following hold:
	\begin{enumerate}
	\item The poset $\X$ satisfies the $\mathbf{W}$ property;
 	\item The action of $G$ on $\X$ is transitive;
	\item The stabilizer in $G$ of some (equivalently: every) $\mathbf{x} \in \X$ is boundedly acyclic.
	\end{enumerate}
	Then $G$ is boundedly acyclic.
\end{theorem}

We will prove this theorem in Section \ref{sec:proof}. In the next section, we will see that this criterion is very powerful in proving bounded acyclicity of transformation groups.

\medskip

There is a corresponding criterion for $\Z$-acyclicity, which will also be proved in Section \ref{sec:proof}.

\begin{theorem}
\label{thm:criterion2}
	Let $G$ be a group acting on a set $X$ and let $\X$ be a non-empty $G$-poset of $X$-sequences. Suppose that the following hold:
	\begin{enumerate}
	\item The poset $\X$ satisfies the $\mathbf{W}$ property;
 	\item The action of $G$ on $\X$ is transitive;
	\item The stabilizer in $G$ of some (equivalently: every) $\mathbf{x} \in \X$ is $\Z$-acyclic.
	\end{enumerate}
	Then $G$ is $\Z$-acyclic.
\end{theorem}

\begin{remark}
\label{rem:w:vs:bac}
    The $\mathbf{W}$-property is simply a convenient, elementary property that occurs very naturally in posets of sequences and implies bounded acyclicity and $\Z$-acyclicity of the poset (Corollary \ref{cor:W}). Therefore more general criteria could be stated, where only the assumption of [bounded] acyclicity is made on the poset. We prefer to state the criterion this way, because it allows us to delay the discussion of bounded cohomology of posets until Section \ref{sec:proof}.
\end{remark}

\subsection{Bounded acyclicity for compactly supported groups}

When applying the criterion from Theorem \ref{thm:criterion}, we will need a way to prove bounded acyclicity of the stabilizers. These will always be \emph{compactly supported} transformation groups, which are known to be boundedly acyclic in great generality \cite{matsumor, binate, lamplighters, ccc}. For instance, to deduce Theorems \ref{thm:main:homeo} and \ref{thm:main:diff} from the criterion, we will only need the following statement:

\begin{theorem}[{\cite[Theorem 1.6, Lemma 2.3]{MN}}]
\label{thm:compactly}
    Let $n \in \N^*, r \in \N^* \cup \{\infty\}$, let $M$ be a closed $C^r$-manifold. Then the groups $\homeoc(M \times \R^n), \diffrc(M \times \R^n)$, are boundedly acyclic.
    
    Moreover, any (possibly infinite, unrestricted) direct product of such groups is boundedly acyclic.
\end{theorem}

In more delicate cases, such as for discs, we will need a more general criterion (which actually implies Theorem \ref{thm:compactly}).

\begin{definition}[{\cite[Definition 4.2]{ccc}}]
    We say that the group $G$ has \emph{commuting $\Z$-conjugates} if for every finitely generated subgroup $H \leq G$ there exists $t \in G$ such that $[H, t^p H t^{-p}] = 1$ for all $p \in \N^*$.
\end{definition}

\begin{theorem}[{\cite[Theorem 1.3]{ccc}}]
\label{thm:ccc}
    If the group $G$ has commuting $\Z$-conjugates, then $G$ is boundedly acyclic.
\end{theorem}

Another useful feature of bounded acyclicity that we will use in the sequel is the following:

\begin{proposition}[{\cite[Proposition 2.4]{MN}, see also \cite[Theorem 4.1.1]{MR}}]
\label{prop:quotient}
	Consider a short exact sequence $1 \to N \to G \to Q \to 1$, where $N$ is boundedly acyclic. Then the quotient $G \to Q$ induces an isomorphism $\hnb(G) \cong \hnb(Q)$ for all $n\in\N$.
\end{proposition}

The criteria presented so far only give bounded acyclicity. To apply Theorem \ref{thm:criterion2} we use another property which guarantees both $\Z$-acyclicity and bounded acyclicity.

\begin{definition}
\label{def:binate}
We say that a group $G$ is \emph{binate} if for every finitely generated subgroup $H \leq G$ there exist a homomorphism $\psi \colon H \to G$ such that $[H, \psi(H)] = 1$ and an element $t \in G$ such that $t^{-1} \psi(h) t = h \psi(h)$ for all $h \in H$.
\end{definition}

\begin{theorem}[\cite{varadarajan, berrick}]
\label{thm:binate:acyclic}
Binate groups are $\Z$-acyclic.
\end{theorem}

\begin{theorem}[\cite{binate}]
\label{thm:binate:bac}
Binate groups are boundedly acyclic.
\end{theorem}

This is a more restrictive property, however it still applies to an important class of groups of compactly supported homeomorphisms. 

\begin{proposition}[\cite{varadarajan, berrick}]
\label{prop:binate:homeo}
    For all $n \geq 1$, the group $\homeoc(\R^n)$ is binate.
\end{proposition}

This implies that the groups $\homeoc(\R^n)$ are $\Z$-acyclic, which was already proved by Mather \cite{mather}. An advantage of using the binate property is that it passes to products:

\begin{lemma}[\cite{binate:products}]
\label{lem:binate:product}
    Any (possibly infinite) direct product of binate groups is binate.
\end{lemma}

Let us remark that there are some other criteria that we could use in this paper to prove [bounded] acyclicity, for example, the binate groups we deal with have the stronger \emph{dissipated} property. We refer the reader to \cite{displacement} for a detailed discussion of various ``displacement criteria'' of this flavor, and their relation to bounded cohomology.

\subsection{Interweaving subsequences}
\label{sec:weaving}

We now present a common framework that will allow us to interweave subsequences in various settings, in order to verify the $\mathbf{W}$ property.

\medskip

Let $X$ be any set endowed with a preorder relation $\sqsubseteq$ and an equivalence relation $\simeq$ contained in $\sqsubseteq$ (i.e. $x\simeq y \Rightarrow x\sqsubseteq y$).
Define $\X$ to be the collection of all cofinal increasing sequences of pairwise inequivalent elements of $X$. Explicitly, this means:

\begin{itemize}
\item $\forall \, n\leq m: x_n\sqsubseteq x_m$;
\item $\forall \, x\in X \,\exists \, n: x\sqsubseteq x_n$;
\item $\forall \, n\neq m: x_n \not\simeq x_m$.
\end{itemize}

Of course we shall only be interested in this definition for examples where there exist cofinal sequences at all. We record the following elementary observation.

\begin{lemma}
In the above setting, the set $\X$ is closed under passage to subsequences.\qed
\end{lemma}

We will prove that the subsequence order relation $\preceq$ has the desired property $\mathbf{W}$ by cyclically weaving subsequences together.

\begin{proposition}\label{prop:abstract:W}
In the above setting, the poset $\X$ has the $\mathbf{W}$ property.
\end{proposition}

\begin{proof}
Let $\mathcal{Q} \subseteq \X$ be a finite subposet, with minimal elements $\mathbf{x}^1, \ldots, \mathbf{x}^k$. We write $\mathbf{x}^j = \{ x^j_n \}_{n \in \N^*}$. The main step in the construction of the various $\mathbf{y}^J$ is the following inductive construction of a single sequence $\mathbf{y} = (y_n)_{n \in \N^*}$.

We start by setting $y_1 \coloneqq x^1_1$. Next, we will choose $y_2$ among the elements of $\mathbf{x}^2$ as follows. Since $\mathcal{Q}$ is a finite set of sequences of pairwise inequivalent elements, there is $q\in\N$ such that $\forall \, \mathbf{x}\in\mathcal{Q}\,\forall \, r> q: x_r\not\simeq y_1$. Since $\mathbf{x}^2$ is increasing and cofinal, we can choose $p> q$ with $y_1 \sqsubseteq x^2_p$ and we define $y_2=x^2_p$.
The general inductive step is analogous: suppose that $y_1, \ldots, y_n$ have been defined, with each $y_i$ an element of $\mathbf{x}^j$ for $j \equiv i \mod k$. Then we choose $y_{n+1}$ to be an element $x^{j+1}_p$ of $\mathbf{x}^{j+1}$ (or $\mathbf{x}^1$ if $j \equiv 0 \mod k$) with the following property. If $q$ is an upper bound for the index of any element of any $\mathbf{x}\in\mathcal{Q}$ that is equivalent to some $y_1, \ldots, y_n$, then we take $p>q$ such that $y_i \sqsubseteq x^{j+1}_p$ (respectively $y_i \sqsubseteq x^{1}_p$) for all $i=1, \ldots, n$.
The sequence $\mathbf{y}$ is by definition an increasing sequence of pairwise inequivalent elements and it is cofinal because it shares a subsequence with a cofinal increasing sequence (any of the $\mathbf{x}^i$).

Turning to the requirements of the $\mathbf{W}$ property, we now define $\mathbf{y}^I$ for $I\subseteq\{1, \ldots, k \}$ to be the subsequence of $\mathbf{y}$ that selects only those indices that are congruent to an element of $I$ modulo $k$; in particular $\mathbf{y}^{\{1, \ldots, k \}} = \mathbf{y}$. Then by definition $\mathbf{y}^I \preceq \mathbf{y}^J$ whenever $I \subseteq J$, and the fact that $\mathbf{y}^I \preceq \mathbf{x}$ if $\mathbf{x}^i \preceq \mathbf{x}$ for all $i \in I$, is the point of the construction above.
\end{proof}

\begin{remark}\label{rem:W:suposet}
Since the above proof verifies the $\mathbf{W}$ property by explicitly interweaving subsequences of the finite set of sequences given for $\mathbf{W}$, it holds verbatim for sub-posets $\mathcal{Y}\subseteq \X$ as soon as $\mathcal{Y}$ has the following two closure properties:

\begin{itemize}
    \item $\mathcal{Y}$ is closed under passing to subsequences;
    \item If $\mathbf{x}\in\X$ can be written as the disjoint union of two subsequences, each in $\mathcal{Y}$, then $\mathbf{x}\in\mathcal{Y}$.
    \end{itemize}
\end{remark}

\section{Verifying the criterion}
\label{sec:verifying}

We now go through the groups from the introduction (except for transformation groups of compact annuli, which will be treated in Section \ref{sec:further}), justifying why they satisfy the criterion from Theorem \ref{thm:criterion} (and in some cases, from Theorem \ref{thm:criterion2} as well). We start with the case of the line, which is the easiest, and the only one among the transformation groups of Euclidean spaces where we are able to establish ordinary $\Z$-acyclicity as well. Then we move on to transformation groups of $\R^n$, and then to discs, where the proof of transitivity will be especially involved. Next, we study the non-compact Cantor set, where the combinatorics of sequences feels different, but still fits into our framework; in this case we will also be able to establish ordinary $\Z$-acyclicity. In the last section, we show that the group of countably supported permutations of an uncountable set satisfies the criteria: although it may seem unrelated, it will feature an important role in the proof itself.

\subsection{Transformation groups of the line}
\label{verifying:line}

In this section we prove that the groups $\diffrp(\R)$ are boundedly acyclic, our proof will depend on Theorem \ref{thm:criterion}. This includes also $\homeop(\R)$ which corresponds to the case $r = 0$. The objects we consider all depend on the regularity, so for simplicity we make the following

\begin{convention}
    We fix $r \in \N \cup \{\infty \}$ for the rest of this section.
\end{convention}

We mostly work with the subgroup $G \leq \diffrp(\R)$ of elements that fix a neighborhood of $(-\infty, 0]$ pointwise. This can be seen as a subgroup of the group of homeomorphisms of $\R_{>0} \coloneqq (0, +\infty)$. The bounded acyclicity of $\diffrp(\R)$ will follow by some standard manipulations.

\medskip

The objects that make up the $G$-poset of $X$-sequences to which Theorem \ref{thm:criterion} applies are the following.

\begin{definition}
\label{def:R:fatpoint}
    A \emph{fat point} in $\R_{>0}$ is a germ at $0$ of an orientation-preserving $C^r$-embedding $(-1, 1) \to \R_{>0}$. The image of $0$ is the \emph{core} $\mathring{x}$ of the fat point $x$.
\end{definition}

In order to define a poset $\fat$ of fat sequences and to deduce from \Cref{prop:abstract:W} that this poset has the $\mathbf{W}$ property, it suffices to define relations $\sqsubseteq$ and $\simeq$ satisfying the condition of Section \ref{sec:weaving}. We define $x\simeq y \Leftrightarrow \mathring x = \mathring y$ and $x\sqsubseteq y \Leftrightarrow \mathring x \leq\mathring y$.
Thus in particular we see that the poset $\fat$ of increasing cofinal sequences of mutually inequivalent elements can be described concretely as the set of \emph{fat sequences}, defined as follows.

\begin{definition}
    A \emph{fat sequence} in $\R_{>0}$ is a sequence $\mathbf{x} = (x_i)_{i \in \N^*}$ of fat points such that the cores $\mathring{\mathbf{x}} = (\mathring{x}_i)_{i \in \N^*}$ form a strictly increasing diverging sequence in $\R_{>0}$.
\end{definition}

The group $G$ acts on fat points by composition: if $x \colon (-1, 1) \to \R_{>0}$ is a fat point with core $\mathring{x}$, then $g \cdot x \coloneqq g \circ x$ is a fat point with core $g \cdot \mathring{x}$. Since the action of $G$ on $\R_{>0}$ is orientation-preserving, the action on sequences of fat points preserves the conditions defining $\fat$. This makes $\fat$ into a $G$-poset of sequences. We now have to verify Items 2 and 3 from Theorem \ref{thm:criterion}:

\begin{lemma}
\label{lem:R:fatsequence:stabilizers}
    For every fat sequence $\mathbf{x} \in \fat$, the stabilizer of $\mathbf{x}$ in $G$ is isomorphic to the power $\diffrc(\R)^{\N}$ and thus it is boundedly acyclic. In case $r = 0$, it is also $\Z$-acyclic.
\end{lemma}

\begin{proof}
Let $\mathbf{x} = (x_i)_{i \in \N^*}$. An element $g \in G$ is in the stabilizer of $\mathbf{x}$ if and only if it fixes a neighborhood of $\mathring{x}_i$ for each $i \in \N^*$. Therefore the stabilizer is isomorphic to the product
\[\diffrc((0, \mathring{x}_1)) \times \prod\limits_{i \geq 2} \diffrc((\mathring{x}_{i-1}, \mathring{x}_i))\]
which is isomorphic to $\diffrc(\R)^{\N}$.
The bounded acyclicity now follows from Theorem \ref{thm:compactly}.
In case $r = 0$, the $\Z$-acyclicity follows from Proposition \ref{prop:binate:homeo}, Lemma \ref{lem:binate:product} and Theorem \ref{thm:binate:acyclic}.
\end{proof}

\begin{lemma}
\label{lem:R:transitivity}
    The action of $G$ on $\fat$ is transitive.
\end{lemma}

\begin{proof}
    Let $\mathbf{x} = (x_i)_{i \in \N^*}$ be the fat sequence whose $i$-th fat point is the germ of the embedding $(-1, 1) \to (i-1, i+1) \subset \R_{>0}$ defined as translation by $i$, so $\mathring{x}_i = i$. Let $\mathbf{y} = (y_i)_{i \in \N^*}$ be any other fat sequence. For each $i$, choose $\varepsilon_i \in (0, \frac{1}{2})$ to be such that the closed intervals $y_i([- \varepsilon_i, \varepsilon_i]) \subset \R_{>0}$ are pairwise disjoint. Then the composition
    \[g_i \colon (i - \varepsilon_i, i + \varepsilon_i) \xrightarrow{x_i^{-1}} (- \varepsilon_i, \varepsilon_i) \xrightarrow{y_i} \R_{>0}\]
    sends the fat point $x_i$ to the fat point $y_i$. Interpolating via cutoff functions matching the regularity of $G$, we can construct $g \in G$ such that $g |_{(i-\varepsilon_i, i+\varepsilon_i)} = g_i$. Then $g \cdot  \mathbf{x} = \mathbf{y}$.
\end{proof}

The transitivity is particularly easy to show in this case, but will require some extra work in higher dimensions.

\medskip

We have thus proved:

\begin{theorem}
\label{thm:criterion:R}

The group $G$ satisfies the criterion from Theorem \ref{thm:criterion}, and thus is boundedly acyclic. In case $r = 0$, the group $G$ also satisfies the criterion from Theorem \ref{thm:criterion2}, and thus is $\Z$-acyclic.
\end{theorem}

This allows to conclude the proof of Theorems \ref{thm:main:homeo} and \ref{thm:main:diff} for the case of the line:

\begin{proof}[Proof of Theorems \ref{thm:main:homeo} and \ref{thm:main:diff} for $n = 1$]

Let $\mathcal{G}_{\pm}$ be the group of germs at ${\pm} \infty$ of $\diffrp(\R)$, which gives the short exact sequence
    \[1 \to \diffrc(\R) \to \diffrp(\R) \to \mathcal{G}_- \times \mathcal{G}_+ \to 1.\]
    Since $\diffrc(\R)$ is boundedly acyclic by Theorem \ref{thm:compactly}, the quotient induces an isomorphism in bounded cohomology, by Proposition \ref{prop:quotient}. Since moreover $\mathcal{G}_+ \cong \mathcal{G}_-$, it suffices to show that $\mathcal{G}_+$ is boundedly acyclic. For this, we consider the short exact sequence
    \[ 1 \to G \cap \diffrc(\R) \to G \to \mathcal{G}_+ \to 1.\]
    Now a homeomorphism $(0, +\infty) \cong (-\infty, +\infty)$ induces an isomorphism $G \cap \diffrc(\R) \cong \diffrc(\R)$, which is therefore boundedly acyclic, so once again the quotient induces an isomorphism in bounded cohomology. We have shown that $G$ is boundedly acyclic, so $\mathcal{G}_+$ is also boundedly acyclic, which concludes the proof.
\end{proof}

We similarly obtain Theorem \ref{thm:acyclic}.

\begin{proof}[Proof of Theorem \ref{thm:acyclic}]
	As in the previous proof, we pass from the $\Z$-acyclicity of $G$ and $\homeoc(\R)$ to the $\Z$-acyclicity of $\homeop(\R)$ using the analogue of Proposition \ref{prop:quotient} in homology; namely, if $1 \to N \to G \to Q \to 1$ is a short exact sequence and $N$ is $\Z$-acyclic, then the quotient $G \to Q$ induces an isomorphism $\HH_n(G; \Z) \cong \HH_n(Q, \Z)$ for all $n \in \N$ (this follows from an application of the Lyndon--Hochschild--Serre spectral sequence \cite[Theorem VII.6.3]{brown}).
\end{proof}

\begin{corollary}
\label{cor:boundedly:Z:acyclic}
    The bounded cohomology of $\homeop(\R)$ with trivial $\Z$-coefficients vanishes in all degrees $n \geq 1$.
\end{corollary}

\begin{proof}
    By \cite[Corollary 5.8]{binate}, this is a consequence of $\Z$-acyclicity (Theorem \ref{thm:acyclic}) and bounded acyclicity (Theorem \ref{thm:main:homeo}).
\end{proof}

\subsection{Transformation groups of Euclidean space}
\label{verifying:Rn}

In this section, we prove Theorems \ref{thm:main:homeo} and \ref{thm:main:diff} for the group $G \coloneqq \diffrp(\R^n)$; this includes also $\homeop(\R^n)$, which corresponds to the case of $r = 0$. Once again, for simplicity we make the following

\begin{convention}
    We fix $r \in \N \cup \{\infty \}$ and $n \in \N_{\geq 2}$ for the rest of this section.
\end{convention}

Our goal is to construct a poset of sequences of \emph{fat spheres} (replacing fat points in higher dimensions), to which we can apply Theorem \ref{thm:criterion}. The construction is close to the one for $\R$, however, it presents some additional difficulties.

\medskip

The main thing that needs addressing is transitivity. In the case of $\R$, we clearly have transitivity on fat points, a basic starting point for transitivity on fat sequences. In our setting, we will need transitivity on germs of embeddings of (locally flat) spheres in $\R^n$. Recall that the local flatness of $K$, a $k$-dimensional submanifold of $n$-dimensional manifold $N$ means that every point $x\in K$ has a neighborhood $U$ in $N$ such that the pair $(U, U\cap K)$ is homeomorphic to $(\R^n, \R^k)$ where $\R^k$ is the first $k$ coordinates of $\R^n$.

The fact that the action on locally flat spheres is transitive is now non-trivial, and would have to rely on the annulus Theorem \cite{Kirby}, which in $C^r$-category for $r>0$ is not known in dimension $4$. To amend this, we work with a poset that only allows for embeddings of a certain type:

\begin{definition}
\label{def:Rn:model:fatsphere}
    The \emph{model fat sphere} in $\R^n$ is the germ at $S^{n-1} \times \{ 0 \}$ of the (orientation-preserving, $C^r$) embedding $\Sigma \colon S^{n-1} \times (-\frac{1}{2}, \frac{1}{2}) \to \R^n$ defined by $(x, \rho) \mapsto (1+\rho)x$.
    
    The \emph{core} of $\Sigma$, denoted by $\mathring{\Sigma}$, is the embedding $S^{n-1} \to \R^n$ obtained by restricting $\Sigma$ to $S^{n-1} \times \{ 0 \}$.
    We denote by $\Sigma_b \coloneqq B^n_1(0)$ and by $\Sigma_u \coloneqq \R^n \setminus \overline{B}^n_1(0)$; these are respectively the \emph{bounded} and \emph{unbounded} components of $\R^n \setminus \im \mathring{\Sigma}$.
\end{definition}

By abuse of notation, we will often identify an embedding with its germ. Given $g \in G$, we define $g \cdot \Sigma \coloneqq g \circ \Sigma$, which is an orientation-preserving $C^r$-embedding $S^{n-1} \times (-\frac{1}{2}, \frac{1}{2}) \to \R^n$, and again we use $g \cdot \Sigma$ to also denote the germ at $S^{n-1} \times \{ 0 \}$ of this embedding. We will consider all such possibilities.

\begin{definition}
\label{def:Rn:fatsphere}
    A \emph{fat sphere} in $\R^n$ is the germ of an orientation-preserving, $C^r$-embedding $S \colon S^{n-1} \times (-\frac{1}{2}, \frac{1}{2}) \to \R^n$ that can be obtained as $g \cdot \Sigma$ for some $g \in G$.

    The \emph{core} of $S$, denoted by $\mathring{S}$, is the embedding $S^{n-1} \to \R^n$ obtained by restricting $S$ to $S^{n-1} \times \{ 0 \}$. Note that if $S = g \cdot \Sigma$ then $\mathring{S} = g \cdot  \mathring{\Sigma}$.

    We denote by $S_b$ and $S_u$ the \emph{bounded} and \emph{unbounded} components of $\R^n \setminus \im \mathring{S}$. If $S = g \cdot \Sigma$, then $S_b = g \cdot \Sigma_b$ and $S_u = g \cdot \Sigma_u$.
\end{definition}

Note that since we defined fat spheres starting from a model fat sphere, everything is well-defined directly, without any need to appeal to the Jordan--Brower or Schoenflies Theorems. Moreover, the set of fat spheres comes endowed with a natural action of $G$, which is transitive by definition.

\medskip

The poset we will be working on will consist of certain sequences of fat spheres. Having a well-defined notion of bounded components, we will be able to formalize the idea of going to infinity.

\begin{definition}
    We say that a fat sphere $S$ in $\R^n$ \emph{englobes} a set $B \subset \R^n$ if $B$ is contained in the bounded component $S_b$. We say that a sequence of fat spheres $(S_i)_{i \in \N^*}$ is \emph{concentric going to infinity} if $S_i$ englobes $\im \mathring{S}_{i-1}$, and every compact subset of $\R^n$ is englobed by $S_i$ for large enough $i$.
\end{definition}

As in the case of the line, we can now define relations which allow us to define a poset of fat sequences and deduce from \Cref{prop:abstract:W} that it has the $\mathbf{W}$ property. 
We define $S\simeq T \Leftrightarrow \mathring S = \mathring T$ and $S\sqsubseteq T$ if either $\mathring S = \mathring T$ or $T$ englobes $\mathring S$.

\begin{definition}
\label{def:Rn:fatsequence}
    A \emph{fat sequence} in $\R^n$ is a collection $\mathbf{S} = (S_i)_{i \in \N^*}$ of fat spheres that are concentric going to infinity. We write $\mathbf{S} \preceq \mathbf{T}$ if $\mathbf{S}$ is a subsequence of $\mathbf{T}$, and denote by $\fat$ the corresponding poset.
\end{definition}

As we did with fat spheres, it will be useful to define a basepoint in the set of fat sequences.

\begin{definition}
\label{def:Rn:model:fatsequence}
    Set $\Sigma_1 \coloneqq \Sigma$ to be the model fat sphere. Define $\Sigma_i$ to be the image of $\Sigma$ under the homothety with ratio $i$; note that the image of the core of $\Sigma_i$ in $\R^n$ is the sphere of radius $i$ around the origin. The sequence $(\Sigma_i)_{i \in \N^*}$ is concentric going to infinity, and so it defines a fat sequence $\mathbf{\Sigma}$, which we call the \emph{model fat sequence}.
\end{definition}

We now verify the Items 2 and 3 from Theorem \ref{thm:criterion}.

\begin{lemma}
\label{lem:Rn:fatsequence:stabilizers}
    The stabilizer of the model fat sequence $\mathbf{\Sigma}$ in $G$ is isomorphic to $\diffrc(\R^n) \times \diffrc(S^{n-1} \times \R)^\N$, and thus is boundedly acyclic.
\end{lemma}

\begin{proof}
    An element $g \in G$ is in the stabilizer of $\mathbf{\Sigma}$ if and only if it fixes a neighborhood of $\im \mathring{\Sigma}_i$ for each $i \in \N^*$. Denote by $B_i$ the open ball of radius $i$ around the origin, that is the bounded component defined by $\Sigma_i$. Then the stabilizer is isomorphic to the product of
    \[\diffrc(B_1) \times \prod\limits_{i \geq 2} \diffrc(B_i \setminus \overline{B}_{i-1}).\]
    Since $B_1$ is $C^r$-diffeomorphic to $\R^n$, and $B_{i+1} \setminus B_i$ is $C^r$-diffeomorphic to $S^{n-1} \times \R$, we obtain the result.
    The last statement follows from Theorem \ref{thm:compactly}.
\end{proof}

The last (and hardest) part is transitivity.

\begin{lemma}
\label{lem:Rn:transitivity}
    The action of $G$ on $\fat$ is transitive.
\end{lemma}

\begin{proof}
    Let $\mathbf{\Sigma} = (\Sigma_i)_{i \in \N^*}$ be the model fat sequence. Let $\mathbf{S} = (S_i)_{i \in \N^*}$ be any other fat sequence. Since the action on the set of fat spheres is transitive by construction, for each $i$ there exists $g_i \in G$ such that $g_i \cdot \Sigma_i = S_i$, and $\varepsilon_i \in (0, \frac{1}{4})$ such that the compact sets $\Sigma_i (S^{n-1} \times [-\varepsilon_i, \varepsilon_i])$ are pairwise disjoint, and the compact sets $g_i \circ \Sigma_i (S^{n-1} \times [-\varepsilon_i, \varepsilon_i])$ are also pairwise disjoint. We want to show that there exists $g \in G$ such that the restriction of $g$ on $\Sigma_i (S^{n-1} \times [-\varepsilon_i, \varepsilon_i])$ coincides with the restriction of $g_i$. Then $g \cdot  \mathbf{\Sigma} = \mathbf{S}$.

Given that fat spheres, by definition, are in the orbit of a model fat sphere, each fat sphere comes with data of a global $C^r$-diffeomorphism in $G$. This makes arguing for transitivity easier, otherwise we would have to use the annulus theorem in the $ C^r$ category. To find $g$ such that $g \cdot  \mathbf{\Sigma} = \mathbf{S}$, we inductively construct $h_k\in G$ such that 
\begin{itemize}
    \item $h_k \cdot \Sigma_i=S_i$ for $i\leq k$.
    \item $h_{k+1}^{-1}\circ h_k$ is the identity in an open neighborhood of the standard ball that bounds $\Sigma_k$. 
\end{itemize}

Then $g$ is the composition of all $h_i$'s. This is well defined since each $x\in \R^n$ eventually lies in the interior of $(\Sigma_k)_b$ for some big $k$ and the second condition assures that $h_j(x)=h_k(x)$ for $j\geq k$.

For $k=1$, we let $h_1=g_1$. Now suppose that $h_k$ is given. To construct $h_{k+1}$, we shall modify $h_k$ by an element in $G$  whose support does not intersect $(S_k)_b$ as follows. We have two different embeddings of the thin annulus $\Sigma_{k+1} (S^{n-1} \times [-\varepsilon_{k+1}, \varepsilon_{k+1}])$ into $\R^n$; one is given by the global $C^r$-diffeomorphism $g_{k+1}$ which gives the fat sphere $S_{k+1}$, and the other is given the global $C^r$-diffeomorphism $h_k$. Recall from Remark (ii) in the introduction that $\diffrp(\R^n) = \diffro(\R^n)$ for all $r$. Therefore $h_k$ and $g_{k+1}$ are isotopic to the identity, and so these two embeddings are isotopic embeddings in $\R^n$. But we need to make sure that these two embeddings are isotopic away from $\Sigma_k$. To show this we need a version of the annulus theorem in our setting that follows from the isotopy extension theorem (for $r>0$ see \cite[Section 8]{MR0448362} and \cite{MR0123338}, and for $r=0$ see \cite{EdwardsKirby}).


\begin{claim}
    Let $W_i$ be the region between $\mathring{S_i}=g_i \circ \Sigma_i (S^{n-1} \times \{0\})$ and $\mathring{S}_{i+1}=g_{i+1} \circ \Sigma_{i+1} (S^{n-1} \times \{0\})$. The manifold $W_i$ is $C^r$-diffeomorphic to $S^{n-1}\times [0,1]$.
\end{claim}

As we shall see in the proof of the claim, the reason that this statement is easier than the annulus theorem is that the two spheres come with data of embeddings that are (ambiently) isotopic.

\begin{proof}[Proof of the claim] Since by radial dilation, we can send $\Sigma_i (S^{n-1} \times \{0\})$ to $\Sigma_{i+1} (S^{n-1} \times \{0\})$, we can assume that there exists $f\in G$ such that $f(\mathring{S}_{i+1})=\mathring{S}_{i}$. To simplify the notation, we shall consider the following situation. Suppose $S(r)$ is the standard sphere of radius $r$ around the origin and for $f\in G$ we know that the sphere $S'=f(S(1))$ is in the interior of the unit ball. We want to prove that the region between $S'$ and $S(1)$ is $C^r$-diffeomorphic to $S^{n-1}\times [0,1]$. Since $f$ is isotopic to the identity, the sphere $S'$ is isotopic to $S(1/2)$. Since spheres are compact, the support of this isotopy is contained inside $S(k)$ for some big $k$. The isotopy extension theorem (for $r>0$ see \cite[Section 8]{MR0448362} and \cite{MR0123338} and for $r=0$ see \cite{EdwardsKirby}) says that the isotopy between $S'$ and $S(1/2)$ can be extended to a path of $C^r$-diffeomorphisms of $D(k)$, the ball of radius $k$ whose boundary is $S(k)$, that is compactly supported (i.e. is the identity near $S(k)$). So there is $g$ in $\diffrc(\mathrm{int}(D(k)))$ such that $g(S')=S(1/2)$.

 Note that the region between $S'$ and $S(1)$ is $C^r$-diffeomorphic to the region between $S'$ and $S(k)$ since we just added the standard annulus between $S(1)$ and $S(k)$ to one of the sphere boundaries $S(1)$. So it is enough to show that the region between $S'$ and $S(k)$ is $C^r$-diffeomorphic to the region between $S(1/2)$ and $S(k)$.  Note that $g$ maps the region between $S'$ and $S(k)$ to the region between $S(1/2)$ and $S(k)$ which is what we wanted. 
\end{proof}

So by the claim, the region $W_{k+1}$ between $g_{k+1}\circ \Sigma_{k+1} (S^{n-1} \times \{0\})$    and $g_{k}\circ \Sigma_{k} (S^{n-1} \times \{0\})$ is $C^r$-diffeomorphic to an annulus $S^{n-1}\times [0,1]$. Note that $h_k\circ \Sigma_{k+1} (S^{n-1} \times \{0\})$ also englobes $g_{k}\circ \Sigma_{k} (S^{n-1} \times \{0\})$, so by the claim the  region $W'_{k+1}$ between them is also $C^r$-diffeomorphic to an annulus $S^{n-1}\times [0,1]$.

 So we can find an isotopy from $g_{k+1}\circ \Sigma_{k+1} (S^{n-1} \times \{0\})$  to a sphere close to $g_{k}\circ \Sigma_{k} (S^{n-1} \times \{0\})$ by pushing along the annulus $W_{k+1}$. Similarly we can isotope $h_k\circ \Sigma_{k+1} (S^{n-1} \times \{0\})$ to the sphere close to $g_{k}\circ \Sigma_{k} (S^{n-1} \times \{0\})$ by pushing along $W'_{k+1}$. Therefore, $g_{k+1}\circ \Sigma_{k+1} (S^{n-1} \times [-\varepsilon_{k+1}, \varepsilon_{k+1}])$ and $h_k\circ \Sigma_{k+1} (S^{n-1} \times [-\varepsilon_{k+1}, \varepsilon_{k+1}])$ are isotopic via an isotopy that does not intersect the sphere $g_{k}\circ \Sigma_{k} (S^{n-1} \times \{0\})$.

 Since the thin annuli $g_{k+1}\circ \Sigma_{k+1} (S^{n-1} \times [-\varepsilon_{k+1}, \varepsilon_{k+1}])$ and $h_k\circ \Sigma_{k+1} (S^{n-1} \times [-\varepsilon_{k+1}, \varepsilon_{k+1}])$ are compact subsets of $\R^n$, the isotopy between them is compactly supported. Therefore, by the isotopy extension theorem, there exists $g\in \diffrc(\mathrm{int}(D(k)))$ for a large enough $k$ such that 
 \begin{itemize}
     \item the support of $g$ does not intersect the disc that bounds the sphere $g_{k}\circ \Sigma_{k} (S^{n-1} \times \{0\})$.
     \item $g$ sends $h_{k}\circ \Sigma_{k+1} (S^{n-1} \times [-\varepsilon_{k+1}, \varepsilon_{k+1}])$ to $g_{k+1}\circ \Sigma_{k+1} (S^{n-1} \times [-\varepsilon_{k+1}, \varepsilon_{k+1}])$.
 \end{itemize}

 So now we define $h_{k+1}$ to be $g\circ h_k$ which completes the induction step and the proof of Lemma \ref{lem:Rn:transitivity}.
\end{proof}

We can now conclude the proof of Theorems \ref{thm:main:homeo} and \ref{thm:main:diff}:

\begin{theorem}
\label{thm:criterion:Rn}

The group $\diffrp(\R^n) : r \geq 0, n \geq 2$ satisfies the criterion from Theorem \ref{thm:criterion}, and thus is boundedly acyclic.
\end{theorem}

\begin{remark}
    A similar strategy might work to study the bounded cohomology of $\diffro(M\times \R^n)$ for a closed manifold $M$. Note that we used the isotopy extension theorem for the proof of Lemma \ref{lem:Rn:transitivity}, which is why we need to explicitly require that diffeomorphisms are isotopic to the identity to hope for a generalization. Since we do not have any application in mind for the possible bounded acyclicity of $\diffro(M\times \R^n)$, we do not pursue it in this paper.
\end{remark}

\subsection{Homeomorphism group of the disc}
\label{verifying:disc}

In this section, we tackle closed discs $D^n$, which we model as the unit disc in $\R^n$. In case $n = 1$, there is no difference between homeomorphisms of the disc and of the line, therefore we assume $n \geq 2$. Our goal is to show Theorem \ref{thm:discs}: the group $G \coloneqq \homeo(D^n, \partial)$ of homeomorphisms of the disc fixing the boundary pointwise is boundedly acyclic (note that since the boundary is fixed, the orientation is automatically preserved). We will construct a poset analogous to the case of $\R^n$, to which Theorem \ref{thm:criterion} can be applied; here controlling the behavior at the boundary will present additional difficulties, and is the reason we only consider homeomorphisms.

\begin{definition}
\label{def:D:model:fatsphere}
	The \emph{model fat sphere} in $D^n$ is the germ at $S^{n-1} \times \{ 0 \}$ of the orientation-preserving $C^0$-embedding $\Sigma \colon S^{n-1} \times (-\frac{1}{4}, \frac{1}{4}) \to D^n$ defined by $(x, \rho) \mapsto (\frac{1}{2} + \rho) x$.
	
	The \emph{core} of $\Sigma$, denoted by $\mathring{\Sigma}$ is the embedding $S^{n-1} \to D^n$ obtained by restricting $\Sigma$ to $S^{n-1} \times \{ 0 \}$.
	
	We denote by $\Sigma_{in} \coloneqq B_{\frac{1}{2}}^n(0)$ and $\Sigma_{out} \coloneqq D^n \setminus \overline{B}_{\frac{1}{2}}^n(0)$, these are respectively the \emph{inner} and \emph{outer} components of $D^n \setminus \im \mathring{\Sigma}$.
\end{definition}

Once again we use the same notation for embeddings and their germs, and denote $g \cdot  \Sigma \coloneqq g \circ \Sigma$.

\begin{definition}
\label{def:D:fatsphere}
	A \emph{fat sphere} in $D^n$ is the germ of an orientation-preserving $C^0$-embedding $S \colon S^{n-1} \times (-\frac{1}{4}, \frac{1}{4}) \to D^n$ that can be obtained as $g \cdot  \Sigma$ for some $g \in G$.
	
	The \emph{core} of $S$, denoted by $\mathring{S}$, is the embedding $S^{n-1} \to D^n$ obtained by restricting $S$ to $S^{n-1} \times \{0\}$. Note that if $S = g \cdot  \Sigma$ then $\mathring{S} = g \circ \mathring{\Sigma}$.
	
	We denote by $S_{in}$ and $S_{out}$ the \emph{inner} and \emph{outer} components of $D^n \setminus \im \mathring{S}$.
\end{definition}

As in the case of $\R^n$ we considered sequences of fat spheres going to infinity, here we will consider sequences of fat spheres going to the boundary. The definition of $\simeq$ and $\sqsubseteq$ is exactly the same; note that the cofinality amounts to asking that every compact subset of \emph{the interior of $D^n$} is englobed by $S_i$ for large enough $i$.

In the case of $\R^n$ we simply defined fat sequences, as sequences of fat spheres that are concentric going to infinity. In the case of discs, there is a further condition that will be needed in order to have continuity at the boundary.

\begin{definition}
\label{def:boundary:embedding}
	Let $(S_i)_{i \in \N^*}$ be a cofinal increasing sequence of pairwise inequivalent fat spheres. We say that $(S_i)_{i \in \N^*}$ has \emph{trivial germ at the boundary} if for every $x \in S^{n-1}$ we have
	\[\lim\limits_{i \to \infty} \mathring{S}_i(x) = \partial(x);\]
	where $\partial \colon S^{n-1} \to D^n$ is the canonical embedding of the boundary.
\end{definition}

Since $G$ is continuous at the boundary, and restricts to the identity, the action of $G$ on sequences of fat spheres preserves the properties of being concentric going to the boundary, and having trivial germ at the boundary. We now have our suitable definition of fat sequences:

\begin{definition}
\label{def:D:fatsequence}
	A \emph{fat sequence} in $D^n$ is a cofinal increasing sequence of pairwise inequivalent fat spheres with trivial germ at the boundary. We denote by $\fat$ the corresponding poset.
	
	Set $\Sigma_1 \coloneqq \Sigma$ to be the model fat sphere. Define $\Sigma_i$ to be the image of $\Sigma$ under a homeomorphism of $D^n$ fixing the boundary that restricts to homothety by $2 (1 - 2^{-i})$ on a neighborhood of $\im \mathring{\Sigma}_1$; note that the image of the core of $\Sigma_i$ in $D^n$ is the sphere of radius $1 - 2^{-i}$ around the origin. The sequence $(\Sigma_i)_{i \in \N^*}$ is concentric going to the boundary and has trivial germ at the boundary, so it defines a fat sequence $\mathbf{\Sigma}$, which we call the \emph{model fat sequence}.
\end{definition}

We now have a $G$-poset of sequences $\fat$, and we need to verify the three conditions from Theorem \ref{thm:criterion}. The first one, the $\mathbf{W}$ property, follows from \Cref{prop:abstract:W} in view of \Cref{rem:W:suposet}.
The stabilizer of the model fat sequence does not split as a direct product to which Theorem \ref{thm:compactly} applies. Still, we will be able to apply Theorem \ref{thm:ccc} to obtain its bounded acyclicity.

\begin{lemma}
\label{lem:D:fatsequence:stabilizers}
    The stabilizer of the model fat sequence $\mathbf{\Sigma}$ in $G$ has commuting $\Z$-conjugates, and thus is boundedly acyclic.
\end{lemma}

\begin{proof}
    An element $g \in G$ belongs to the stabilizer $G_{\mathbf{\Sigma}}$ if and only if it fixes a neighborhood of $\mathring{\Sigma}_i$ for each $i \in \N^*$. Denote by $B_i$ the open ball of radius $(1 - 2^{-i})$ around the origin, that is the inner component defined by $\Sigma_i$. Then $G_{\mathbf{\Sigma}}$ is a subgroup of the product
    \[\homeoc(B_1) \times \prod\limits_{i \geq 2} \homeoc(B_i \setminus \overline{B}_{i-1}),\]
    which in turn can be identified with a subgroup of homeomorphisms of the interior of $D^n$. However, unlike the case of $\R^n$, this time $G_{\mathbf{\Sigma}}$ is a proper subgroup, since given a sequence of elements in $B_i \setminus \overline{B}_{i-1}$, their image under every element $g \in G$ must have the same limit in the boundary.
    
    Let $H \leq G_{\mathbf{\Sigma}}$ be a finitely generated subgroup. For $i \geq 2$, denote by $H_i$ the projection of $H$ onto $\homeoc(B_i \setminus \overline{B}_{i-1})$. This is a finitely generated subgroup of a group of compactly supported homeomorphisms of a non-compact space, so all of $H_i$ is supported on some compact set. We can choose this to be a compact annulus $A_i \subset B_i \setminus \overline{B}_{i-1}$. More explicitly, we can find a compact interval $I_i \subset (1 - 2^{1-i}, 1 - 2^{-i})$ such that $A_i$ is the set of all elements of norm contained in $I_i$, and every element of $H_i$ fixes $(B_i \setminus \overline{B}_{i-1}) \setminus A_i$.
    Similarly, we denote by $H_1$ the projection of $H$ on $\homeoc(B_1)$, which is supported on some compact set $A_1 \subset B_1$.
    
    We will construct an element $t \in G_{\mathbf{\Sigma}}$ with the property that $t^p(A_i) \cap A_i = \emptyset$ for all $p, i \in \N^*$. Let us conclude the proof assuming the existence of such an element. For every $p \in \N^*$, the sets $\cup_i A_i$ and $\cup_i t^p(A_i)$ are disjoint. Indeed, by construction $A_i$ is disjoint from $t^p(A_i)$, and moreover $A_i \subset B_i \setminus \overline{B}_{i-1}$ is disjoint from $t^p(A_j) \subset B_j \setminus \overline{B}_{j-1}$ for all $j \neq i$, because $t \in G_{\mathbf{\Sigma}}$ preserves the sets $B_i \setminus \overline{B}_{i-1}$. It follows that the group $H$, supported on $\cup_i A_i$, and the group $t^p H t^{-p}$, supported on $\cup_i t^p(A_i)$, commute.
    This concludes the proof that $G_{\mathbf{\Sigma}}$ has commuting $\Z$-conjugates; the last statement follows from Theorem \ref{thm:ccc}.
    
\medskip
    
    It remains to construct $t$. For $i \geq 2$, let $I_i \coloneqq (a_i, b_i)$, and choose an element $f_i \in \homeoc(1 - 2^{1-i}, 1 - 2^{-i})$ with the property that $f_i(a_i) > b_i$. Since $f_i$ is orientation-preserving, it follows that $f_i^p(I_i) \cap I_i = \emptyset$ for all $p \in \N^*$. We then define an element $t_i \in \homeoc(B_i \setminus \overline{B}_{i-1})$ in cylindrical coordinates by 
    \[S^{n-1} \times (1 - 2^{1-i}, 1 - 2^{-i}) \to S^{n-1} \times (1 - 2^{1-i}, 1 - 2^{-i}) \colon (x, \rho) \mapsto (x, f_i(\rho)).\]
    The choice of $f_i$ then implies that $t_i^p(A_i) \cap A_i = \emptyset$ for all $p \in \N^*$. We also choose an element $t_1 \in \homeoc(B_1)$ such that $t_1^p(A_1) \cap A_1 = \emptyset$ for all $p \in \N^*$.
    
    We define $t$ on the interior of $D^n$ so that $t|_{B_1} = t_1$ and $t|_{B_i \setminus \overline{B}_{i-1}} = t_i$ for all $i \geq 2$, and set $t|_{\partial D^n} = \id|_{\partial D^n}$. Clearly $t$ is a homeomorphism when restricted to the interior of $D^n$, which restricts to the identity on a neighborhood of $\im \mathring{S}_i$ for all $i \in \N^*$, and it satisfies $t^p(A_i) \cap A_i = \emptyset$ for all $p, i \in \N^*$. So it only remains to show that $t \in G$, i.e. $t$ is continuous at the boundary. For this, let $(y_j)_{j \in \N^*}$ be a sequence of points in $D^n$ such that $y_j \to y \in \partial D^n$. Since $t$ restricts to the identity on the boundary, we may assume without loss of generality that $y_j$ belongs to the interior of $D^n$, for all $j$. Up to removing finitely many terms, we may assume that $y_j$ has norm at least $\frac{1}{2}$, which allows to use cylindrical coordinates to check continuity: say $y_j = (x_j, \rho_j), y = (x, 1)$ and so $x_j \to x \in S^{n-1}$ and $\rho_j \to 1 \in [\frac{1}{2}, 1]$. Let $i_j \in \N^*$ be such that $\rho_j \in [1 - 2^{1-i_j}, 1 - 2^{-i_j}]$. Then $t(y_j) = t(x_j, \rho_j) = (x_j, f_j(\rho_j)) \in \{x_j\} \times [1 - 2^{1-i_j}, 1 - 2^{-i_j}]$. In particular $\rho_j \to 1$ implies $f_j(\rho_j) \to 1$, and so $t(y_j) \to y$, which concludes the proof.
\end{proof}

We are left to show transitivity, which follows from the results in the appendix, in particular a controlled version of the annulus theorem.

\begin{lemma}
\label{lem:D:transitivity}
    The action of $G$ on $\fat$ is transitive.
\end{lemma}

\begin{proof}
Combine Theorem \ref{thm:homeo-dd} with Corollary \ref{cor:add-germs}.
\end{proof}

We can now conclude the proof of Theorem \ref{thm:discs}:

\begin{theorem}
\label{thm:criterion:discs}

The group $\homeo(D^n, \partial)$ satisfies the criterion from Theorem \ref{thm:criterion}, and thus is boundedly acyclic.
\end{theorem}

We deduce from this Corollary \ref{cor:discs}:

\begin{proof}[Proof of Corollary \ref{cor:discs}]
	We have a short exact sequence
	\[1 \to \homeo(D^n, \partial) \to \homeop(D^n) \to \homeop(S^{n-1}) \to 1,\]
	where the quotient is given by the restriction. It is indeed surjective, as can be seen by a cone construction. By Theorem \ref{thm:discs} the kernel is boundedly acyclic, and so the conclusion follows from Proposition \ref{prop:quotient}.
\end{proof}

The bounded cohomology of $\homeop(S^1)$, and thus of $\homeop(D^2)$, was fully computed in \cite[Theorems 1.1 and 1.3]{MN}. In higher dimensions, the bounded cohomology of $\homeop(S^{n-1})$ has been computed in low degrees \cite[Theorems 1.8 and 1.10]{MN}, and using Corollary \ref{cor:discs} we obtain:

\begin{corollary}
\label{cor:disc:low}
    For all $n \geq 3$, the bounded cohomology of $\homeop(D^n)$ vanishes in degrees $2$ and $3$. Moreover, the bounded cohomology of $\homeop(D^4)$ vanishes in degree $4$.
\end{corollary}

\subsection{Homeomorphism group of the non-compact Cantor set}
\label{verifying:cantor}

The non-compact Cantor set $\mathbf{K}$ is uniquely characterized as being a space that is non-compact, locally compact, totally disconnected, metrizable and without isolated points. Concretely, it can be modelled as $K \times \N^*$, where $\N^*$ is the set of positive natural numbers with the discrete topology, and $K$ is a Cantor set, for example in the dyadic model $2^{\N^*}$. We denote by $G$ the homeomorphism group of $\mathbf{K}$. We fix a basepoint $o$ of $K$.

Following \cite{andritsch}, we say that a \emph{fat point} in $\mathbf{K}$ is the germ at $o$ of an embedding $x \colon K \to \mathbf{K}$ with (compact) open image. The \emph{core} of the embedding is $\mathring{x} \coloneqq x(o)$. We are again in the abstract setting of \Cref{sec:weaving} if we define $x\simeq y \Leftrightarrow \mathring x = \mathring y$ and this time take for $\sqsubseteq$ the trivial relation ($x\sqsubseteq y \, \forall \, x,y$). We consider pairwise inequivalent sequences defined by requiring that the sequence of cores form a closed discrete subset of $\mathbf{K}$. We call these simply \emph{fat sequences}, and denote by $\fat$ the corresponding subposet. This subposet satisfies the requirements of \Cref{rem:W:suposet} and therefore $\fat$ satisfies the $\mathbf{W}$ property. 

\begin{lemma}
    The action of $G$ on $\fat$ is transitive.
\end{lemma}

\begin{proof}
    Let $\mathbf{x} = (x_i)_{i \in \N^*}$ and $\mathbf{y} = (y_i)_{i \in \N^*}$ be two fat sequences. For each $i \in \N$ we choose a compact open neighborhood $U_i$ of $\mathring{x}_i$ in such a way that the $U_i$ are pairwise disjoint. Because the sequence $(\mathring{x}_i)_{i \in \N^*}$ is discrete, by the topological characterizations of $K$ and $\mathbf{K}$, each $U_i$ is homeomorphic to $K$, and each $\mathbf{K} \setminus U_i$ is homeomorphic to $\mathbf{K}$. Thus, up to enlarging each $U_i$, we may assume that $(U_i)_{i \in \N^*}$ forms a partition of $\mathbf{K}$ by countably many compact open sets. 
    
    We choose similarly a partition $(V_i)_{i \in \N^*}$ for $\mathbf{y}$ (note that the two partitions are allowed to overlap). By \cite[Lemma 3.3.16]{andritsch}, for each $i$ there exists a homeomorphism $h_i \colon U_i \to V_i$ that sends $x_i$ to $y_i$. The unique map $h \colon \mathbf{K} \to \mathbf{K}$ that agrees with $h_i$ on $U_i$ for every $i \in \N^*$ is then a homeomorphism that sends $\mathbf{x}$ to $\mathbf{y}$.
\end{proof}

\begin{lemma}
    The stabilizers for the action of $G$ on $\fat$ are binate, therefore they are $\Z$-acyclic and boundedly acyclic.
\end{lemma}

\begin{proof}
    We fix the model $\mathbf{K} = K \times \N^*$. By transitivity, we can consider the simplest fat sequence, namely $\mathbf{x} = (x_i)_{i \in \N^*}$, where the embedding $x_i \colon K \to \mathbf{K} = K \times \N$ is defined by $x_i(p) = (p, i)$. Let $H$ be a finitely generated subgroup of the stabilizer $G_{\mathbf{x}}$. Then for each $i \in \N^*$, there exists a compact open neighborhood $U_i$ of $o \in K$ such that $H$ fixes pointwise $U_i \times \{i\}$, which is a compact open neighborhood of $\mathring{x}_i = (o, i) \in \mathbf{K}$. Let $S_i$ denote the complement of $U_n$ in $K$, and let $V_i$ be a compact open subset of $K$ such that $o \in V_i \subsetneq U_i$. Finally, we let $W_i \coloneqq U_i \setminus V_i$, which is again a compact open subset of $K$.
    
    Now $V_i, W_i$ and $S_i$ are all homeomorphic to $K$, and $V_i \sqcup W_i \sqcup S_i = K$, so we have a decomposition:
    \[\mathbf{K} = \left( \bigcup_i V_i \times \{ i \} \right) \bigsqcup \left( \bigcup_i W_i \times \{ i \} \right) \bigsqcup \left( \bigcup_i S_i \times \{ i \} \right) \eqqcolon \mathbf{K}_V \sqcup \mathbf{K}_W \sqcup \mathbf{K}_S.\]
    Each of $\mathbf{K}_V, \mathbf{K}_W$ and $\mathbf{K}_S$ is homeomorphic to $\mathbf{K}$, the group $H$ is supported on $\mathbf{K}_S$, and the fat sequence $\mathbf{x}$ is represented by embeddings taking values in $\mathbf{K}_V$.

    By the topological characterization of $\mathbf{K}$, a countably infinite disjoint union of copies of $\mathbf{K}$ is homeomorphic to $\mathbf{K}$. Therefore for a non-compact Cantor set $\mathbf{K}_0$, there are homeomorphisms $\mathbf{K}_W \to \mathbf{K}_0 \times (\Z \setminus \{0\})$ and $\mathbf{K}_S \to \mathbf{K}_0 \times \{0\}$. Together, these define a homeomorphism $\mathbf{K}_W \sqcup \mathbf{K}_S \to \mathbf{K}_0 \times \Z$, which gives a new model
    \[
    \mathbf{K} = \mathbf{K}_V \sqcup \left( \mathbf{K_0} \times \Z \right).
    \]
    The group $H$ is supported on $\mathbf{K}_0 \times \{ 0 \}$.
    
    We now define a homeomorphism $t \colon \mathbf{K} \to \mathbf{K}$ as follows:
    \[
    t(x) \coloneqq \begin{cases}
        x \text{ if } x \in \mathbf{K}_V; \\
        (x_0, i+1) \text{ if } x = (x_0, i) \in \mathbf{K}_0 \times \{ i \} \subset \mathbf{K}_0 \times \Z.
    \end{cases}
    \]
    Then $t \in G$, and since $t$ restricts to the identity on $\mathbf{K}_V$, it belongs to the stabilizer $G_{\mathbf{x}}$.

    It remains to define the homomorphism $\psi \colon H \to G$. We set
    \[\psi(h)(x) \coloneqq \begin{cases}
        x \text{ if } x \in \mathbf{K}_V \sqcup (\mathbf{K}_0 \times \Z_{\leq 0}) \\
        t^i h t^{-i}(x) \text{ if } x \in \mathbf{K}_0 \times \{ i \} \subset \mathbf{K}_0 \times \Z_{>0}.
    \end{cases}\]
    Since $H$ is supported on $\mathbf{K}_0 \times \{ 0 \}$ and $\psi(H)$ is supported on $\mathbf{K}_0 \times \Z_{> 0}$, we have $[H, \psi(H)] = 1$. Moreover, $t^{-1} \psi(h)t = h \psi(h)$ for all $h \in H$, which follows by a direct computation. This proves that $G_{\mathbf{x}}$ is binate, and we conclude using Theorems \ref{thm:binate:acyclic} and \ref{thm:binate:bac}.
\end{proof}

This concludes the proof of Theorem \ref{thm:cantor}:

\begin{theorem}
\label{thm:cantor:criterion}
    The group $\homeo(\mathbf{K})$ satisfies the criteria from Theorems \ref{thm:criterion} and \ref{thm:criterion2}, and thus is $\Z$-acyclic and boundedly acyclic.
\end{theorem}

\subsection{Countably supported permutations of an uncountable set}
\label{ss:auxiliary}

Our last example is of a different flavor from the previous ones, but it will be crucial in the proofs of Theorems \ref{thm:criterion} and \ref{thm:criterion2}.

\begin{definition}
Let $X$ be an uncountable set. We let $G$ be the group of permutations of $X$ that have countable support.
\end{definition}

In contrast to the previous results in this section, we will not use the criteria from Theorems \ref{thm:criterion} and \ref{thm:criterion2} to prove that $G$ is [boundedly] acyclic; this is because the [bounded] acyclicity of $G$ will be used in the proof of the criterion itself. Instead, we establish the [bounded] acyclicity $G$ directly:

\begin{proposition}
\label{prop:auxiliary:bac}
The group $G$ is binate, therefore it is $\Z$-acyclic and boundedly acyclic.
\end{proposition}

\begin{proof}
Let $H \leq G$ be a finitely generated subgroup. Because each element of $H$ has countable support, there exists a countably infinite subset $Y_0 \subset X$ such that $H$ is supported on $Y_0$. Now let $Y_i, i \in \Z \setminus \{ 0 \}$ be countably infinite subsets of $X$ that are disjoint from $Y_0$ and from each other. Fix a bijection $\pi_i \colon Y_{i-1} \to Y_i$ for all $i \in \Z$. We define
\[
t(x) \coloneqq
\begin{cases}
\pi_i(x) \text{ if } x \in Y_{i-1}; \\
x \text{ otherwise}.
\end{cases}
\]
Then $t$ is supported on the union of the $Y_i$, which is countable, so $t \in G$. Next, for each $i \geq 1$ let $\Pi_i \coloneqq \pi_i \circ \cdots \circ \pi_1 \colon Y_0 \to Y_i$. We define a homomorphism $\psi \colon H \to G$ as follows:
\[
\psi(h)(x) \coloneqq
\begin{cases}
\Pi_i h \Pi_i^{-1}(x) \text{ if } x \in Y_i; \\
x \text{ otherwise}.
\end{cases}
\]
Then $\psi(h)$ is supported on $\cup_{i \geq 1} Y_i$, so it belongs to $G$. It is easy to see that $[H, \psi(H)] = 1$ and that $t^{-1} \psi(h) t = h \psi(h)$. This shows that $G$ is binate, the last statements then follow from Theorems \ref{thm:binate:acyclic} and \ref{thm:binate:bac}.
\end{proof}

We now show that $G$ satisfies the conditions for our two criteria.

\begin{proposition}
\label{prop:criterion:auxiliary}

The group $G$ satisfies the criteria from Theorems \ref{thm:criterion} and \ref{thm:criterion2}.
\end{proposition}

\begin{proof}
Let $\X$ be the set of sequences of pairwise distinct elements in $X$. Then $\X$ is a $G$-poset of sequences, which satisfies the $\mathbf{W}$ property by Proposition \ref{prop:abstract:W} by taking $\simeq$ to be equality and $\sqsubseteq$ to be the trivial relation. Because sequences are countable, the action is transitive. The stabilizer $G_{\mathbf{x}}$ of a sequence $\mathbf{x}$ with image $Y$ is the subgroup of $G$ supported on $X \setminus Y$; since this set has the same cardinality as $X$, it follows that $G_{\mathbf{x}}$ is isomorphic to $G$, and thus is $\Z$-acyclic and boundedly acyclic by Proposition \ref{prop:auxiliary:bac}.
\end{proof}

\section{Proof of the criterion}
\label{sec:proof}

In this section, we prove the criteria from Theorems \ref{thm:criterion} and \ref{thm:criterion2}. Most of the section will be devoted to the proof of Theorem \ref{thm:criterion}, which yields the bounded acyclicity results that are the main subject of this paper, we will show how to adapt the arguments to prove Theorem \ref{thm:criterion2} in Section \ref{proof:acyclic}.

\medskip

We start in Section \ref{proof:ss} with some background on the bounded cohomology of semisimplicial sets. This will include a result from \cite{MN} which allows to compute the bounded cohomology of a group via an action on a boundedly acyclic semisimplicial set with boundedly acyclic stabilizers. In Section \ref{proof:poset} we examine in detail the case of posets, and see how the $\mathbf{W}$ property implies bounded acyclicity; this will require the proof of several basic results about the cohomology of posets to be made uniform, in the spirit of \cite{matsumor, simplicial}.

In Section \ref{proof:orbit} we examine the orbit complex for the action of a group as in Theorem \ref{thm:criterion}. This is quite straightforward, but it is perhaps the most surprising part of the argument: it will imply that a group as in Theorem \ref{thm:criterion} has the same bounded cohomology as the monoid $\emb(\N^*)$ of order-preserving embeddings of $\N^*$. Finally, in Section \ref{proof:monoid} we show that $\emb(\N^*)$ is boundedly acyclic, this will be proved indirectly, by appealing to the bounded acyclicity of the group of countably supported permutations of an uncountable set, which we treated in Section \ref{ss:auxiliary}.

\subsection{Bounded cohomology of semisimplicial sets}
\label{proof:ss}

Let $X_\bullet$ be a semisimplicial set: as in \cite{MN}, it is sufficient to work in this context, as we will never need to use the additional structure of a simplicial set. Its \emph{bounded cohomology} $\hnb(X_\bullet)$ is the cohomology of the complex
\[ 0 \to \ell^\infty(X_0) \to \ell^\infty(X_1) \to \ell^\infty(X_2) \to \cdots\]
where the coboundary operators are defined as the duals of the face maps of $X_\bullet$. The coefficients are always understood to be $\R$ unless specified otherwise. We say that $X_\bullet$ is \emph{boundedly acyclic} if it has the bounded cohomology of a point, namely it is connected and $\hnb(X_\bullet) = 0$ for all $n > 0$.

If $G$ is a discrete group and $X_\bullet$ is the nerve of $G$, this defines the \emph{bounded cohomology} of $G$, denoted by $\hnb(G)$. More explicitly, we have $X_\bullet = G^\bullet$, with face maps defined as
\begin{align*}
    d_i& \colon X_n \to X_{n-1} \\
    d_0&(g_1, \ldots, g_n) = (g_2, \ldots, g_n); \\
    d_i&(g_1, \ldots, g_n) = (g_1, \ldots, g_{i-1}, g_i g_{i+1}, g_{i+2}, \ldots, g_n), \quad i \in \{1, \ldots, n-1\}; \\
    d_n&(g_1, \ldots, g_n) = (g_1, \ldots, g_{n-1}).
\end{align*}
This definition can be directly generalized to define the bounded cohomology of a discrete \emph{semigroup} (or monoid), which will appear in Section \ref{proof:monoid}. In both cases, we say that the (semi)group $G$ is \emph{boundedly acyclic} if $\hnb(G) = 0$ for all $n > 0$.

\medskip

In the case of groups, the homogeneous resolution identifies the bounded cohomology of $G$ with the cohomology of the subcomplex of invariants:
\[0 \to \ell^\infty(G)^G \to \ell^\infty(G^2)^G \to \ell^\infty(G^3)^G \to \cdots\]
In this context, we have a semisimplicial set $X_\bullet$, which is nothing but the simplex with vertex set $G$ with the usual face maps, endowed with a simplicial action of $G$; and the subcomplex of invariants is the complex computing the bounded cohomology of the semisimplicial set of orbits $G \backslash X_\bullet$. Simplicial actions of $G$ can calculate its bounded cohomology in much greater generality:

\begin{theorem}[{\cite[Theorem 3.3]{MN}}, see also \cite{MR}]
\label{thm:resolutions}
    Let $G$ be a group acting on a boundedly acyclic semisimplicial set $X_\bullet$. Suppose that, for all $p \in \N$:
    \begin{enumerate}
        \item The stabilizer of each point in $X_p$ is boundedly acyclic;
        \item There are only finitely many isomorphism types of such stabilizers.
    \end{enumerate}
    Then there is an isomorphism $\hnb(G) \cong \hnb(G \backslash X_\bullet)$.
\end{theorem}

The second condition is stronger than what is required in \cite[Theorem 3.3]{MN}, but this is the only version we need, and it requires less terminology to state.

\begin{remark}
    This is a result specifically about trivial real coefficients (it follows from \cite[Proposition 1.2]{addendum:MR} that it also applies to all dual trivial coefficients). Since it is used at the beginning of the proof, we cannot say anything about the bounded cohomology with non-trivial coefficients of the groups considered here. This is contrary to the compactly supported case, where one can usually deal with trivial and general separable dual coefficients in a unified way \cite{lamplighters, ccc}.
\end{remark}

In order to prove bounded acyclicity of the relevant semisimplicial sets, we will use a framework, introduced by Matsumoto and Morita \cite{matsumor}, and recently developed further by Kastenholz and Sroka \cite{simplicial}. Given a semisimplicial set $X_\bullet$, we let $\Ctilde(X_\bullet)$ denote the corresponding augmented simplicial chain complex, endowed with the $\ell^1$-norm with respect to the standard basis.

\begin{definition}
    Let $f, g \colon X_\bullet \to Y_\bullet$ be two simplicial maps between semisimplicial sets. A \emph{bounded homotopy} between $f$ and $g$ is a sequence of bounded linear maps $\{ h_p \colon \Ctilde_p(X_\bullet) \to \Ctilde_{p+1}(Y_\bullet)\}_{p \geq 0}$ such that
    \[d_1 h_0 = f_0 - g_0\]
    and for all $p \in \N^*$:
    \[d_{p+1} h_p + h_{p-1} d_p = f_p - g_p.\]
    If a bounded homotopy exists, we say that $f$ and $g$ are \emph{uniformly homotopic}.
    
    Two semisimplicial sets are \emph{uniformly homotopy equivalent} if there exist simplicial maps $f \colon X_\bullet \to Y_\bullet$ and $g \colon Y_\bullet \to X_\bullet$ such that $fg$ and $gf$ are both uniformly homotopic to the identity.
\end{definition}

Now we consider the notion of uniform acyclicity for semisimplicial sets, which is what we will use to prove [bounded] acyclicity in the sequel.

\begin{definition}
    We say that a semisimplicial set $X_\bullet$ is \emph{uniformly acyclic} if the following holds.
    For every $p \geq 0$ there exists a constant $K_p^{X_\bullet} > 0$ such that for every cycle $z \in \Ctilde_p(X_\bullet)$ there exists $c \in \Ctilde_{p+1}(X_\bullet)$ such that $d_{p+1}(c) = z$ and $\| c \| \leq K_p^{X_\bullet} \cdot \| z \|$.
\end{definition}

This is relevant to us because of the following result, already implicit in \cite{matsumor}:

\begin{lemma}[{\cite[Corollary 7.4]{simplicial}}]
\label{lem:uac:bac}
    If $X_\bullet$ is uniformly acyclic, then it is boundedly acyclic.
\end{lemma}

In order to check uniform acyclicity of a semisimplicial set, it can be useful to restrict to its finite sub-semisimplicial sets.

\begin{definition}
    Let $f \colon X_\bullet \to Y_\bullet$ be a simplicial map. We say that $f$ is \emph{uniformly acyclic} if the following holds. For every $p \geq 0$ there exists a constant $K_p^f > 0$ such that for every cycle $z \in \Ctilde_p(X_\bullet)$ there exists $c \in \Ctilde_{p+1}(Y_\bullet)$ such that $d_{p+1}(c) = f_p(z)$ and $\| c \| \leq K_p^f \cdot \| z \|$.
\end{definition}

\begin{lemma}
\label{lem:finite}
    Let $Y_\bullet$ be a semisimplicial set. Suppose that there exist constants $K_p > 0$ such that every inclusion of a finite sub-semisimplicial set $\iota \colon X_\bullet \hookrightarrow Y_\bullet$ is uniformly acyclic with constants $K_p^\iota \leq K_p$. Then $Y_\bullet$ is uniformly acyclic, with constants bounded by $K_p$.
\end{lemma}

\begin{proof}
    This is a direct consequence of the fact that uniform acyclicity is defined in terms of simplicial cycles, which are supported on finitely many simplices.
\end{proof}

The easiest case of a uniformly acyclic semisimplicial set is that of a cone.

\begin{lemma}[{\cite[Corollary 7.4]{simplicial}}]
\label{lem:cone}
    A simplicial cone $CX_\bullet$ over the semisimplicial set $X_\bullet$ is uniformly acyclic, and the constant $K_p^{CX_\bullet}$ can be chosen to be equal to $1$.
\end{lemma}

We rephrase this in the relative setting of uniformly acyclic maps:

\begin{lemma}
\label{lem:cone:inclusion}
    Let $f \colon X_\bullet \to Y_\bullet$ be a simplicial map that factors through the simplicial cone $CX_\bullet$. Then $f$ is uniformly acyclic, and the constant $K_p^f$ can be chosen to be equal to $1$.
\end{lemma}

These notions are compatible with bounded homotopy equivalences.

\begin{lemma}
\label{lem:htpy:acyclic}
    Let $f, g \colon X_\bullet \to Y_\bullet$ be two simplicial maps that are uniformly homotopic via the bounded homotopy $\{ h_p \}_{p \geq 0}$. If $f$ is uniformly acyclic, then so is $g$, with constants $K_p^g \leq K_p^f + \| h_p \|$.
\end{lemma}

\begin{proof}
    Let $z \in \Ctilde_p(X_\bullet)$ be a reduced cycle. Since $f$ is uniformly acyclic, there exists $c \in \Ctilde_{p+1}(Y_\bullet)$ such that $d_{p+1}(c) = f_p(z)$ and $\| c \| \leq K_p^f \cdot \| z \|$. Now
    \[g_p(z) = d_{p+1}(h_p(z)) + h_{p-1}(d_p(z)) - f_p(z) = d_{p+1}(h_p(z) - c),\]
    where we used that $z$ is a reduced cycle and the definition of $c$. Moreover
    \[\| h_p(z) - c \| \leq \| h_p \| \cdot \| z \| + K_p^f \cdot \| z \|. \qedhere\]
\end{proof}

\subsection{Bounded cohomology of posets}
\label{proof:poset}

Every poset $(\PO, \preceq)$ has an associated semisimplicial set: its \emph{nerve}, where a $p$-simplex is a chain $\{ x_0 \preceq x_1 \preceq \cdots \preceq x_p \}$, and the face maps correspond to deleting an element from the chain. Let us stress that this definition allows for degenerate simplices, which is why we are working with semisimplicial sets instead of simplicial complexes.

We will denote posets by $\PO$, and the corresponding semisimplicial sets by $\PO_\bullet$. Posets are more amenable to proofs of acylicity than general semisimplicial sets \cite{bjoerner}, and some of those techniques can be adapted to bounded cohomology \cite{simplicial}. In this section, we prove bounded analogues of some basic lemmas in this context, which will lead to a useful combinatorial criterion for uniform acyclicity of a poset. We start with the \emph{acyclic carrier lemma} \cite[Lemma 10.1]{bjoerner}.

\begin{definition}
    Let $X_\bullet, Y_\bullet$ be semisimplicial sets. A \emph{carrier} is a map $\Phi$ which to each simplex $\sigma \in X_p$ associates a sub-semisimplicial set $\Phi(\sigma)$ of $Y_\bullet$, and such that $\Phi(\sigma) \subseteq \Phi(\tau)$ whenever $\sigma \subseteq \tau$.
    A carrier $\Phi$ is said to be \emph{uniformly acyclic} if for all $\sigma \in X_\bullet$, the semisimplicial set $\Phi(\sigma)$ is uniformly acyclic, and the corresponding constants $K_p^{\Phi(\sigma)}$ in degree $p$ are bounded by constants $K_p$ independent of $\sigma$.

    A \emph{carrier for a simplicial map $f \colon X_\bullet \to Y_\bullet$} is a carrier $\Phi$ such that $f(\sigma) \in \Phi(\sigma)$ for all $\sigma \in X_\bullet$.    
\end{definition}

\begin{lemma}
\label{lem:carrier}
    Let $f, g \colon X_\bullet \to Y_\bullet$ be two simplicial maps. Suppose that $f, g$ admit a common uniformly acyclic carrier $\Phi$. Then $f$ and $g$ are uniformly homotopic. More precisely, if the constants $K_p$ witness the uniform acyclicity of $\Phi$, and satisfy $1 \leq K_p \leq K_{p+1}$, then the bounded homotopy $\{ h_p \}_{p \geq 0}$ can be chosen to satisfy $\| h_p \| \leq 2^{p+1}(p+1)!(K_p)^{p+1}$.
\end{lemma}

The proof will show that a sharper bound is possible, even unconditionally, but this is all we will need for the proofs.

\begin{proof}
    We construct the bounded homotopy $\{ h_p \}_{p \geq 0}$ by induction. For $\sigma \in X_0$, both $f_0(\sigma)$ and $g_0(\sigma)$ represent an element in $\Phi(\sigma)_0$, and so $z \coloneqq (f_0 - g_0)(\sigma)$ is a reduced cycle in $\Ctilde_0(\Phi(\sigma))$. By the uniform acyclicity assumption, there exists a chain $c \in \Ctilde_1(\Phi(\sigma))$ such that $d_1(c) = z$ and $\| c \| \leq K_0 \cdot \| z \| \leq 2 K_0$. We define $h_0(\sigma) \coloneqq c$, extend it linearly to a map $h_0 \colon \Ctilde_0(X_\bullet) \to \Ctilde_1(Y_\bullet)$ and notice that $d_1 h_0 = f_0 - g_0$ and $\| h_0 \| \leq 2 K_0 = 2^1 1! K_0^1$.

    Now let $p \in \N^*$ and suppose by induction that for all $0 \leq i < p$ bounded maps $h_i \colon \Ctilde_i(X_\bullet) \to \Ctilde_{i+1}(Y_\bullet)$ have been defined so that the homotopy identity and the bound on $\|h_i\|$ hold, and so that moreover $h_i(\sigma)$ is supported on $\Phi(\sigma)$ for all $\sigma \in X_i$. Fix $\sigma \in X_p$, and set $z \coloneqq (f_p - g_p - h_{p-1} d_p)(\sigma)$. Then the induction hypothesis shows that $d_p(z) = 0$ and $z$ is supported on $\Phi(\sigma)$, so $z \in \Ctilde_p(\Phi(\sigma))$ is a cycle. By the uniform acyclicity assumption, there exists $c \in \Ctilde_{p+1}(\Phi(\sigma))$ such that $d_{p+1}(c) = z$ and $\| c \| \leq K_p \cdot \| z \| \leq K_p \cdot (2 + (p+1)\| h_{p-1} \|)$. We define $h_p(\sigma) \coloneqq c$, extend it linearly to a map $h_p \colon \Ctilde_p(X_\bullet) \to \Ctilde_{p+1}(Y_\bullet)$, and notice that $d_{p+1} h_p - h_{p-1} d_p = f_p - g_p$, and
    \[\| h_p \| \leq K_p (2 + (p+1)\| h_{p-1} \|) \leq 2K_p + K_p (p+1)! 2^p K_{p-1}^p \leq 2^{p+1}(p+1)! K_p^{p+1}. \qedhere\]
\end{proof}

This uniformly acyclic carrier lemma can be used to prove a bounded version of Quillen's order homotopy theorem \cite[Theorem 10.11]{bjoerner}. The latter could also be proved more directly, but we took the opportunity to state and prove the former, which is useful beyond the present application \cite{flatmates}.

\begin{theorem}
\label{thm:quillen}
    Let $X_\bullet$ be a semisimplicial set, and let $\PO$ be a poset. Let $f, g \colon X_\bullet \to \PO_\bullet$ be simplicial maps. If $f(x) \preceq g(x)$ for every $x \in X_0$, then $f$ and $g$ are uniformly homotopic via a bounded homotopy $\{ h_p \}_{p \geq 0}$ with $\| h_p \| \leq 2^{p+1}(p+1)!$.
\end{theorem}

\begin{proof}
    For each simplex $\sigma \in X_\bullet$ define $\Phi(\sigma)$ to be the sub-semisimplicial set of $\PO_\bullet$ spanned by $f(\sigma) \cup g(\sigma)$. Every vertex of $\Phi(\sigma)$ dominates the minimal element of $f(\sigma)$, so the corresponding semisimplicial set is a cone, which is uniformly acyclic and whose constant can be chosen to be equal to $1$ in all degrees (Lemma \ref{lem:cone}). The result then follows from Lemma \ref{lem:carrier}.
\end{proof}

We will only use the following consequence of Theorem \ref{thm:quillen} (the $\mathbf{W}$ property is given in Definition \ref{def:W}).

\begin{corollary}
\label{cor:W}
    Let $\PO$ be a poset with the $\mathbf{W}$ property. Then $\PO_\bullet$ is $\Z$-acyclic and uniformly acyclic.
\end{corollary}

Note that we defined uniform acyclicity for the real chain complex, so it immediately implies $\R$-acyclicity, but not ($\Z$-)acyclicity.

\begin{proof}
    We start with uniform acyclicity, for which we will show that $\PO_\bullet$ satisfies the criterion of Lemma \ref{lem:finite}. Up to replacing a finite sub-semisimplicial set with the span of its vertices, we just need to show that for every finite subposet $\mathcal{Q} \subset \PO$, the inclusion $\iota \colon \mathcal{Q}_\bullet \to \PO_\bullet$ is uniformly acyclic with constants independent of $\mathcal{Q}$. Let $x_1, \ldots, x_k$ be the minimal elements of $\mathcal{Q}$, and let $y_I, I \subseteq \{ 1, \ldots, k\}$ be the elements of $\PO$ whose existence is ensured by the $\mathbf{W}$ property. Define the map $f \colon \mathcal{Q} \to \PO$ by sending $x \in \mathcal{Q}$ to $y_{I_x}$, where $I_x = \{i \in \{ 1, \ldots, k \} \mid x_i \preceq x \}$. Item 1 of Definition \ref{def:W} states that $y_I \preceq y_J$ if $I \subseteq J$, and this implies that $f$ is order-preserving. Item 2 states that $y_{I_x} \preceq x$, and so  $f(x) \preceq x$, for all $x \in \mathcal{Q}$. Therefore Theorem \ref{thm:quillen} applies and shows that $f \colon \mathcal{Q}_\bullet \to \PO_\bullet$ is uniformly homotopic to $\iota$, via a bounded homotopy $\{ h_p \}_{p \geq 0}$ with $\| h_p \| \leq 2^{p+1}(p+1)!$. However, every element of $f(\mathcal{Q})$ is dominated by $y_{\{1, \ldots, k\}}$, so $f(\mathcal{Q}_\bullet)$ is contained in a cone. Lemma \ref{lem:cone:inclusion} then shows that $f$ is uniformly acyclic with $K_p^f \leq 1$ for all $p$. Finally, we apply Lemma \ref{lem:htpy:acyclic} to deduce that $\iota$ is uniformly acyclic with $K_p^\iota \leq K_p^f + \| h_p \| \leq 1 + 2(p+1)$. Since this bound is independent of $\mathcal{Q}$, we conclude from Lemma \ref{lem:finite} that $\PO$ is uniformly acyclic.

    The proof can be adapted verbatim to the case of integral coefficients (without keeping track of the norms) using \cite[Theorem 10.11]{bjoerner} instead of Theorem \ref{thm:quillen}. This proves that $\PO_\bullet$ is also $\Z$-acyclic.
\end{proof}

\subsection{The orbit complex}
\label{proof:orbit}

Recall the setup for Theorem \ref{thm:criterion}: the group $G$ is acting on a set $X$, and there is a non-empty $G$-invariant poset of sequences in $X$, denoted $\X$, such that the action of $G$ on $\X$ is transitive, and has boundedly acyclic stabilizers. Moreover, $\X$ satisfies the $\mathbf{W}$ property, which implies that the semisimplicial set $\X_\bullet$ is boundedly acyclic, by Corollary \ref{cor:W} (this is the only way in which we are using the $\mathbf{W}$-property, see also Remark \ref{rem:w:vs:bac}). Therefore the action of $G$ on $\X_\bullet$ satisfies the hypotheses of Theorem \ref{thm:resolutions}. It remains to show that the orbit complex $G \backslash \X_\bullet$ is boundedly acyclic. In this section we study the orbit complex, and for this analysis we only need to assume the transitivity of the action.

\medskip

Consider a simplex $\{ \mathbf{x}_0 \preceq \mathbf{x}_1 \preceq \cdots \preceq \mathbf{x}_p \}$ for $p \geq 0$. Define $\iota_i \colon \mathbf{x}_i \to \N^*$ to be the index map, which associates to a fat point in $\mathbf{x}_i$ its index in the sequence.
Now $\mathbf{x}_{p-1}$ is uniquely determined by the indices that describe it as a subsequence of $\mathbf{x}_p$. Therefore, once $\mathbf{x}_p$ has been given, $\mathbf{x}_{p-1}$ can be described in terms of an \emph{order-preserving embedding} $\eta_p \colon \N^* \hookrightarrow \N^*$, such that the following diagram commutes:
\[\begin{tikzcd}
	{\mathbf{x}_{p-1}} & {} & {\mathbf{x}_p} \\
	\\
	\N^* && \N^*
	\arrow["{\eta_p}", hook, from=3-1, to=3-3]
	\arrow[hook, from=1-1, to=1-3]
	\arrow["{\iota_p}", from=1-3, to=3-3]
	\arrow["{\iota_{p-1}}"', from=1-1, to=3-1]
\end{tikzcd}\]
Proceeding inductively in the same way, we define order-preserving embeddings $\eta_i \colon \N^* \to \N^*$ that intertwine the inclusion of $\mathbf{x}_{i-1}$ into $\mathbf{x}_i$ with the order isomorphisms $\iota_{i-1}, \iota_i$, so that the following diagram commutes:
\[\begin{tikzcd}
	{\mathbf{x}_0} & {\mathbf{x}_1} & \cdots & {\mathbf{x}_{p-1}} & {\mathbf{x}_p} \\
	\N^* & \N^* & \cdots & \N^* & \N^*
	\arrow[hook, from=1-1, to=1-2]
	\arrow[hook, from=1-2, to=1-3]
	\arrow[hook, from=1-3, to=1-4]
	\arrow[hook, from=1-4, to=1-5]
	\arrow["{\iota_p}", from=1-5, to=2-5]
	\arrow["{\iota_{p-1}}"', from=1-4, to=2-4]
	\arrow["{\iota_1}", from=1-2, to=2-2]
	\arrow["{\iota_0}"', from=1-1, to=2-1]
	\arrow["{\eta_1}"', hook, from=2-1, to=2-2]
	\arrow["{\eta_2}"', hook, from=2-2, to=2-3]
	\arrow["{\eta_{p-1}}"', hook, from=2-3, to=2-4]
	\arrow["{\eta_p}"', hook, from=2-4, to=2-5]
\end{tikzcd}\]

We denote by $\emb(\N^*)$ the monoid of order-preserving embeddings of $\N^*$. We write multiplication in $\emb(\N^*)$ in terms of the right action of $\emb(\N^*)$ on $\N^*$, so given $\eta, \zeta \in \emb(\N^*)$, their product $\eta \zeta$ is the embedding given by applying $\eta$, then $\zeta$. Let $I_p \colon \X_p \to \emb(\N^*)^p$ be the map that associates to a $p$-simplex $\{ \mathbf{x}_0 \preceq \cdots \preceq \mathbf{x}_p \}$ the $p$-tuple $(\eta_1, \ldots, \eta_p)$.

\begin{lemma}
\label{lem:orbit:bij}
    The map $I_p \colon \X_p \to \emb(\N^*)^p$ induces a bijection $G \backslash \X_p \to \emb(\N^*)^p$, which we also denote by $I_p$.
\end{lemma}

\begin{proof}
    Given a $p$-simplex $\{ \mathbf{x}_0 \preceq \cdots \preceq \mathbf{x}_p \}$, the action of $G$ intertwines the maps $\iota_i$, i.e. the following diagram commutes:
    \[\begin{tikzcd}
	{\mathbf{x}_i} && {g \cdot \mathbf{x}_i} \\
	\\
	\N^* && \N^*
	\arrow["{=}", from=3-1, to=3-3]
	\arrow["{g \cdot }", from=1-1, to=1-3]
	\arrow["{\iota_p}", from=1-3, to=3-3]
	\arrow["{\iota_p}"', from=1-1, to=3-1]
    \end{tikzcd}\]
    Since this is true for each $i$, we see that $I_p$ is invariant under the action of $G$, hence it descends to a map $I_p \colon G \backslash \X_p \to \emb(\N^*)^p$. This map is surjective: for a sequence of embeddings $(\eta_1, \ldots, \eta_p)$, we can define a $p$-simplex $\{ \mathbf{x}_0 \preceq \cdots \preceq \mathbf{x}_p \}$, by picking an arbitrary $\mathbf{x}_p$, and inductively defining $\mathbf{x}_{i-1}$ as the subsequence of $\mathbf{x}_i$ that realizes the embedding $\eta_i$. So it remains to show that $I_p \colon G \backslash \X_p \to \emb(\N^*)^p$ is injective.

    Let $\{ \mathbf{x}_0 \preceq \cdots \preceq \mathbf{x}_p \}$ and $\{ \mathbf{y}_0 \preceq \cdots \preceq \mathbf{y}_p \}$ be two $p$-simplices with the same image under $I_p$. Let $g \in G$ be an element such that $g \cdot \mathbf{x}_p = \mathbf{y}_p$; this exists by the transitivity assumption. Since the embedding $\eta_p \colon \N^* \to \N^*$ is valid for both $g \cdot \mathbf{x}_{p-1} \to g \cdot \mathbf{x}_p$ and $\mathbf{y}_{p-1} \to \mathbf{y}_p$, the indices that define $\mathbf{y}_{p-1}$ as a subsequence of $\mathbf{y}_p$ are the same that describe $g \cdot \mathbf{x}_{p-1}$ as a subsequence of $g \cdot \mathbf{x}_p = \mathbf{y}_p$. This shows that $g \cdot \mathbf{x}_{p-1} = \mathbf{y}_{p-1}$, and repeating the same argument by induction we conclude that these two $p$-simplices are in the same $G$-orbit.
\end{proof}

This produces an identification of the simplices of the orbit complex. We now show that this extends to an isomorphism:

\begin{lemma}
\label{lem:orbit:iso}
    The maps $I_p \colon G \backslash \X_p \to \emb(\N^*)^p$ define an isomorphism of semisimplicial sets $G \backslash \X_\bullet \to \emb(\N)^\bullet$ from the orbit complex to the nerve of $\emb(\N)^\bullet$.
\end{lemma}

\begin{proof}
    We already know that $I_p$ is a bijective map sending simplices to simplices. We need to show that these maps intertwine the face maps. Let $\{ \mathbf{x}_0 \preceq \cdots \preceq \mathbf{x}_p \}$ be a simplex, whose image under $I_p$ is $(\eta_1, \ldots, \eta_p)$. Removing $\mathbf{x}_i$ yields the following commutative diagrams. For $0 < i < p$:
\[\begin{tikzcd}
	{\mathbf{x}_0} & {\mathbf{x}_1} & \cdots & {\mathbf{x}_{i-1}} & {\mathbf{x}_{i+1}} & \cdots & {\mathbf{x}_{p-1}} & {\mathbf{x}_p} \\
	\N^* & \N^* & \cdots & \N^* & \N^* & \cdots & \N^* & \N^*
	\arrow[hook, from=1-1, to=1-2]
	\arrow[hook, from=1-2, to=1-3]
	\arrow[hook, from=1-3, to=1-4]
	\arrow[hook, from=1-4, to=1-5]
	\arrow["{\iota_{i+1}}", from=1-5, to=2-5]
	\arrow["{\iota_{i-1}}"', from=1-4, to=2-4]
	\arrow["{\iota_1}"', from=1-2, to=2-2]
	\arrow["{\iota_0}"', from=1-1, to=2-1]
	\arrow["{\eta_1}"', hook, from=2-1, to=2-2]
	\arrow["{\eta_2}"', hook, from=2-2, to=2-3]
	\arrow["{\eta_{i-1}}"', hook, from=2-3, to=2-4]
	\arrow["{\eta_i \eta_{i+1}}"', hook, from=2-4, to=2-5]
	\arrow[hook, from=1-5, to=1-6]
	\arrow["{\eta_{i+2}}"', from=2-5, to=2-6]
	\arrow["{\eta_{p-1}}"', hook, from=2-6, to=2-7]
	\arrow[hook, from=1-6, to=1-7]
	\arrow["{\iota_{p-1}}"', from=1-7, to=2-7]
	\arrow["{\iota_p}", from=1-8, to=2-8]
	\arrow[hook, from=1-7, to=1-8]
	\arrow["{\eta_p}"', hook, from=2-7, to=2-8]
\end{tikzcd}\]
    While for $i = 0$ and $i = p$ we are simply removing the first and last columns, respectively. We now compute the effect on face maps, starting with $i = 0$:
    \begin{align*}
         I_{p-1}(d_p^0(\{\mathbf{x}_0 \preceq \cdots \preceq \mathbf{x}_p\})) &= I_{p-1}(\{\mathbf{x}_1 \preceq \cdots \preceq \mathbf{x}_p\}) \\
         &= (\eta_2, \ldots, \eta_p) \\
         &= d_p^0(\eta_1, \ldots, \eta_p) = d_p^0(I_p(\{\mathbf{x}_0 \preceq \cdots \preceq \mathbf{x}_p\})).
    \end{align*}
    Analogously:
    \[I_{p-1}(d_p^n(\{\mathbf{x}_0 \preceq \cdots \preceq \mathbf{x}_p\})) = d_p^n(I_p(\{\mathbf{x}_0 \preceq \cdots \preceq \mathbf{x}_p\})).\]
    And for $0 < i < p$:
    \begin{align*}
         I_{p-1}(d_p^i(\{\mathbf{x}_0 \preceq \cdots \preceq \mathbf{x}_p\})) &= I_{p-1}(\{\mathbf{x}_1 \preceq \cdots \preceq \mathbf{x}_{i-1} \preceq \mathbf{x}_{i+1} \preceq \cdots \preceq \mathbf{x}_p\}) \\
         &= (\eta_2, \ldots, \eta_{i-1}, \eta_i \eta_{i+1}, \eta_{i+2}, \ldots, \eta_p) \\
         &=d_p^i(\eta_1, \ldots, \eta_p) = d_p^i(I_p(\{\mathbf{x}_0 \preceq \cdots \preceq \mathbf{x}_p\})). \qedhere
    \end{align*}
\end{proof}

The discussion so far is summarized in the following result.

\begin{proposition}
\label{prop:orbitcomplex}
    Let $G$ be a group acting on a set $X$ and let $\X$ be a non-empty $G$-poset of sequences in $X$. Suppose that the action of $G$ on $\X$ is transitive. Then there is an isomorphism in bounded cohomology $\hnb(\emb(\N^*)) \cong \hnb(G \backslash \X_\bullet)$.
\end{proposition}

Combined with the other assumptions in Theorem \ref{thm:criterion}, we obtain the following.

\begin{proposition}
\label{prop:reduction}
    Let $G$ be a group satisfying the criterion from Theorem \ref{thm:criterion}. Then there is an isomorphism in bounded cohomology $\hnb(\emb(\N^*)) \cong \hnb(G)$.
\end{proposition}

\begin{proof}
    Consider the action of $G$ on $\X_\bullet$. This semisimplicial set is uniformly acyclic by Corollary \ref{cor:W}, so it is boundedly acyclic by Lemma \ref{lem:uac:bac}. The stabilizer of a simplex $\{ \mathbf{x}_0 \preceq \cdots \preceq \mathbf{x}_p \}$ coincides with the stabilizer of $\mathbf{x}_p$, which is boundedly acyclic, and all such stabilizers are conjugate by transitivity. So Theorem \ref{thm:resolutions} applies and gives an isomorphism $\hnb(G) \cong \hnb(G \backslash \X_\bullet)$. By Proposition \ref{prop:orbitcomplex}, the latter is in turn isomorphic to $\hnb(\emb(\N^*))$.
\end{proof}

\subsection{Bounded cohomology of the embedding monoid}
\label{proof:monoid}

It remains to show that the embedding monoid $\emb(\N^*)$ is boundedly acyclic. This will be proved indirectly, using the group of countably supported permutations on an uncountable set, from Section \ref{ss:auxiliary}.

\begin{theorem}
\label{thm:emb:bac}

The monoid $\emb(\N^*)$ is boundedly acyclic.
\end{theorem}

\begin{proof}
Proposition \ref{prop:reduction} shows that if a group satisfies the criterion from Theorem \ref{thm:criterion}, then its bounded cohomology is isomorphic to the bounded cohomology of $\emb(\N^*)$. Therefore it suffices to exhibit a group that satisfies the criterion from Theorem \ref{thm:criterion}, that is moreover boundedly acyclic. The group of countably supported permutations of an uncountable set satisfies the criterion by Proposition \ref{prop:criterion:auxiliary}, and it is boundedly acyclic by Proposition \ref{prop:auxiliary:bac}, whose proof does not rely on Theorem \ref{thm:criterion}.
\end{proof}

\begin{proof}[Proof of Theorem \ref{thm:criterion}]
Combine Proposition \ref{prop:reduction} and Theorem \ref{thm:emb:bac}.
\end{proof}

\subsection{Acyclicity}
\label{proof:acyclic}

A completely analogous proof holds in the case of $\Z$-acyclicity. As before, we split the proof into two statements.

\begin{proposition}
\label{prop:reduction2}
    Let $G$ be a group satisfying the criterion from Theorem \ref{thm:criterion2}. Then there is an isomorphism in homology $\HH_n(\emb(\N^*); \Z) \cong \HH_n(G; \Z)$.
\end{proposition}

\begin{proof}
    The analogue of Theorem \ref{thm:resolutions} holds for homology with integral coefficients, in our case. Indeed, consider the action of $G$ on $\X_\bullet$, which is $\Z$-acyclic by Corollary \ref{cor:W}. Then, by \cite[Chapter VII.7]{brown} there is a spectral sequence
    \[E^1_{p, q} = \bigoplus_{\sigma \in \Sigma_p} \HH_p(G_\sigma; \Z_\sigma) \Rightarrow \HH_{p+q}(G; \Z),\]
    where $G_\sigma$ is the stabilizer of the simplex $\sigma$, and $\Z_\sigma$ is the module $\Z$ with the $\Z[G_\sigma]$-module structure defined by the sign of the permutation that an element of $G_\sigma$ induces on the set of vertices of $\sigma$. In our setting, the semisimplicial sets in questions are posets, so $G_\sigma$ fixes the vertices of $\sigma$ pointwise, and thus we recover the isomorphism $\HH_p(G; \Z) \cong \HH_p(G \backslash \X_\bullet; \Z)$.
    We conclude by Lemma \ref{lem:orbit:iso}.
\end{proof}

\begin{theorem}
\label{thm:emb:acyclic}

The monoid $\emb(\N^*)$ is $\Z$-acyclic.
\end{theorem}

\begin{proof}
As in the proof of Theorem \ref{thm:emb:bac}, we appeal to Proposition \ref{prop:reduction2} and the $\Z$-acyclicity of the group of countably supported permutations of an uncountable set (Proposition \ref{prop:auxiliary:bac}).
\end{proof}

\begin{proof}[Proof of Theorem \ref{thm:criterion2}]
Combine Proposition \ref{prop:reduction2} and Theorem \ref{thm:emb:acyclic}.
\end{proof}

\begin{remark}
    As we have seen in the previous section, actions on posets of sequences with \emph{boundedly acyclic} stabilizers are easier to construct than actions with $\Z$-\emph{acyclic} stabilizers. However, let us point out that the isotropy spectral sequence that we used in the proof of Proposition \ref{prop:reduction2} only requires the $\Z$-acyclicity of the semisimplicial set being acted on, which in our setting follows from Corollary \ref{cor:W}. A more careful analysis of this spectral sequence for the groups in question, aided by the $\Z$-acyclicity of the orbit complex established via Theorem \ref{thm:emb:acyclic} (which holds assuming only transitivity on sequences, see Proposition \ref{prop:orbitcomplex}), could shed some light on the homology of the transformation groups covered in Section \ref{sec:verifying}, even when we know that $\Z$-acyclicity does not hold.
\end{remark}

\section{Unboundedness of the topological Pontryagin classes}
\label{sec:pontryagin}

Recall from the introduction that the map
\[
B\homeo(M)^\delta\to B\homeo(M),
\]
is an acyclic map when $M$ is a compact manifold or the interior of a compact manifold with boundary (\cite[Section 2, Theorem 2.5]{mcduff1980homology}). The same holds for $\homeop(M)$. If we combine this result for $M=\R^n$ and the non-vanishing result of Galatius and Randal-Williams \cite{galatius2022algebraic} we obtain that 
 \[
 \R[e,p_1, p_2, \dots]\to \HH^*(B\homeop(\R^{2n})^\delta;\R),
 \]
 is injective for $2n\geq 6$, and in the odd-dimensional case, it follows from \cite[Corollary 1.2]{galatius2022algebraic} that
  \[
 \R[p_1, p_2, \dots]\to \HH^*(B\homeop(\R^{2n+1})^\delta;\R),
 \]
 is injective for $2n+1\geq 7$. In dimension $2$ and $3$, we obtain that 
 $\R[p_1]\to \HH^*(B\homeop(\R^3)^\delta;\R)$ and $\R[e]\to \HH^*(B\homeop(\R^2)^\delta;\R)$ are isomorphisms. Given that $\homeop(\R^n)$ is a boundedly acyclic group by Theorem \ref{thm:main:homeo}, all the non-trivial classes in $\HH^*(B\homeop(\R^n)^\delta;\R)$ are unbounded. This gives a proof of Corollary \ref{cor:flat R^n bundles}. In general, for oriented $C^r$-flat $\R^n$-bundles  for any regularity $r$, we have the following result.

 \begin{theorem}\label{C^r unboundedness}
     For any regularity $r\geq 0$, the polynomials in terms of the Euler class (when $n$ is even) and Pontryagin classes $p_i$ in $H^*(B\diffrp(\R^n)^\delta; \R)$ that are of degree less than $n+2$ have no bounded representatives, and for $r=1$ one can even drop the condition on the degree of such polynomials. 
 \end{theorem}
\begin{proof}
    
By a deep result of Segal \cite[Propositions~1.3 and~3.1]{segal1978classifying}, we know that there is a natural map $$B\diffrp(\R^n)^\delta\to B\mathrm{S}\Gamma^{r}_{n},$$ which is a homology isomorphism, where $B\mathrm{S}\Gamma^{r}_{n}$ is the classifying space of Haefliger structures for codimension $n$ foliations that are transversely oriented, see~\cite[Section~1]{segal1978classifying} for more details. The homotopy type of the classifying space $B\mathrm{S}\Gamma^{r}_{n}$ is a mystery but we know its homotopy type in the following range of degrees by comparing it to the classifying space $B\mathrm{GL}_n(\R)_{+}$. There is a map 
\[
\nu\colon B\mathrm{S}\Gamma^{r}_{n}\to B\mathrm{GL}_n(\R)_{+},
\]
which classifies oriented normal bundles to the codimension $n$ foliations where $\mathrm{GL}_n(\R)_{+}$ is the group of invertible matrices with positive determinant. For all regularities, it is known that the map $\nu$ is at least $(n+1)$-connected, see~\cite[Remark~1, Section~II.6]{haefliger1971homotopy}. Hence, in particular, the induced map 
\[
\HH^*(B\mathrm{GL}_n(\R)_{+})\to \HH^*(B\diffrp(\R^n)^\delta),
\]
is an isomorphism for $*\leq n$ and is injective for $*=n+1$. So the polynomials of the Euler class, when $n$ is even, and the Pontryagin classes $p_i$ are non-trivial for all $r$ when the degree of the polynomial is less than $n+2$. Given that $\diffrp(\R^n)$ is a boundedly acyclic group by Theorem \ref{thm:main:diff}, all these non-trivial classes in $\HH^*(B\diffrp(\R^n)^\delta)$ are unbounded.

For $r=1$, it is a consequence of a remarkable theorem of Tsuboi \cite{tsuboi1989foliated} that the cohomology of $B\diff^1(M)$, the classifying space of the group $\diff^1(M)$ with $C^1$-topology, is the same as the group cohomology of $\diff^1(M)$ as a discrete group. Since $\diff^1_+(\R^n)\simeq \mathrm{GL}_n(\R)_{+}$, we deduce that 
\[
\HH^*(B\mathrm{GL}_n(\R)_{+})\to \HH^*(B\diff^1_+(\R^n)^\delta),
\]
is an isomorphism in all degrees and the rest of the argument is the same as before. 
\end{proof}
To prove the unboundedness of Pontryagin classes for $C^0$-flat sphere bundles, we need the following lemma.

\begin{lemma}\label{lem:stabilization}
    Let $\mathrm{st}\colon \homeo(\R^n)\to\homeo(\R^{n+1})$ be the stabilization map that sends a homeomorphism $f$ to $f\times \mathrm{id}$. Let $\alpha$ be the following composition
    \[
    \homeo(\R^n)\to \homeo(S^n)\to \homeo(D^{n+1})\to\homeo(\mathrm{int}(D^{n+1}))\xrightarrow{\cong} \homeo(\R^{n+1}),
    \]
    where the first map is induced by the one-point compactification, the second map is induced by taking a cone over the origin, the third map is the restriction to the interior of the disc which is homeomorphic to $\R^{n+1}$. By choosing an orientation-preserving homeomorphism between $\mathrm{int}(D^{n+1})$ and $\R^{n+1}$, we obtain a continuous homomorphism $\alpha\colon \homeo(\R^n)\to \homeo(\R^{n+1})$ which is independent of the choice up to homotopy. Then also the maps $\mathrm{st}$ and $\alpha$ are homotopic through homomorphisms.
\end{lemma}

\begin{proof}
 First note that different identifications are induced by conjugation by an orientation-preserving homomorphism of $\R^n$. Since conjugation by elements that are isotopic to the identity induces a homomorphism on $\homeo(\R^n)$ that is homotopic to the identity, the homomorphism $\alpha$ is well defined up to homotopy.
 
 We identify $S^n$ as the one-point compactification of $\R^n$ and choose $\infty_n$, the point at infinity, as the base point of $S^n$. Let $CS^n$ be the  cone $S^n\times [0,1]/S^n\times \{0\}$ and let $\mathring{C}S^n$ be the open cone $S^n\times [0,1)/S^n\times \{0\}$. Let also $SS^n=CS^n/S^n\times\{1\}$ be the unreduced suspension of $S^n$. Let us denote the path in $CS^n$ and $SS^n$ which is the image of $\infty_n\times [0,1]$ by $l$. We identify $S^{n+1}$ with $SS^n$ and the point $\infty_{n+1}$ with the image of $S^n\times \{1\}$.

For a closed subset $Y$ of a space $X$, we write $\homeo(X,Y)$ to denote those homomorphisms of $X$ that restrict to the identity on $Y$. We can identify $\homeo(\R^n)$ with $\homeo(S^n, \infty_n)$. So we consider the following model of the stabilization map
\begin{equation*}
\resizebox{.99\hsize}{!}{$
 \homeo(S^n, \infty_n)\to \homeo(S^n\times [0,1], \infty_n\times [0,1])\to \homeo(SS^n, l)\to \homeo(S^{n+1}, \infty_{n+1}),$}
\end{equation*}
meaning after identifying the left end by $\homeo(\R^n)$ and the right end by $\homeo(\R^{n+1})$, we obtain a homotopic map to the standard stabilization map.

  On the other hand, we can identify $\mathrm{int}(D^{n+1})$ with $\mathring{C}S^n$ in such a way that the line $l$ is identified with the line from the origin to $\infty_n\in S^n$. With this identification, we have a model of the homomorphism $\alpha$ as the following composition
  \[
   \homeo(S^n, \infty_n)\to \homeo(\mathring{C}S^n, l)\to\homeo(SS^n, l)\to \homeo(S^{n+1}, \infty_{n+1}),
  \]
 so it is homotopic to $\alpha$. But this model is the same homomorphism as the above model for the stabilization map. Therefore, $\alpha$ and the standard stabilization map are homotopic. 
\end{proof}

Using this lemma we can prove the unboundedness of certain Pontryagin classes for flat sphere bundles as follows.
    \begin{proof}[Proof of Theorem \ref{thm:spherebundles}] First we shall see the non-vanishing of topological Pontryagin classes in $\HH^*(B\homeo(S^n)^\delta;\R)$. By Lemma \ref{lem:stabilization} and the Thurston--McDuff homology isomorphism \cite{mcduff1980homology}, we know that the composition
    \[
    \HH^*(B\homeo(\R^{n})^\delta;\R)\to \HH^*(B\homeo(S^{n-1})^\delta; \R)\to \HH^*(B\homeo(\R^{n-1})^\delta;\R),
    \]
    is induced by the stabilization map. By Galatius--Randal-Williams \cite{galatius2022algebraic} we know that 
    \[
    \R[p_1,p_2,\dots]\hookrightarrow  \HH^*(B\homeo(\R^n)^\delta;\R),
    \]
    is injective for all $n\geq 6$. Therefore, for $n\geq 7$, we have the compatible injection of $ \R[p_1,p_2,\dots]$ to the source and the target of the above composition. So, for $n\geq 7$ we also have the injective map
    \[
    \R[p_1,p_2,\dots]\hookrightarrow  \HH^*(B\homeo(S^{n-1})^\delta;\R).
    \]
    Now consider the following commutative diagram

    \begin{tikzpicture}[node distance=5cm, auto]
  \node (A) {$\HH^*(B\homeo(S^{n})^\delta; \R)$};
  \node (B) [right of=A] {$\HH^*(B\homeo(\R^{n})^\delta; \R)$};
  \node (E) [right of=B] {$\HH^*(B\homeo(S^{n-1})^\delta; \R)$};
  \node (C) [below of=A, node distance=1.2cm] {$\hsb(B\homeo(S^{n})^\delta; \R)$};  
  \node (D) [below of=B, node distance=1.2cm] {$\hsb(B\homeo(\R^{n})^\delta; \R)$};
   \node (F) [below of=E, node distance=1.2cm] {$\hsb(B\homeo(S^{n-1})^\delta; \R).$};

  \draw[->] (C) to node {} (D);
    \draw[->] (D) to node {} (F);
  \draw[->] (B) to node {} (E);
  \draw[<-] (E) to node {} (F);
  \draw [<-] (A) to node {}(C);
  \draw [->] (A) to node {} (B);
  \draw [<-] (B) to node {} (D);
\end{tikzpicture}

By Theorem \ref{thm:main:homeo}, we know that $\hsb(B\homeo(\R^{n})^\delta; \R)=0$ in positive degrees. A diagram chase implies that all the classes $\R[p_1,p_2,\dots]$ in $\HH^*(B\homeo(S^{n})^\delta;\R)$ are unbounded.

For $4\leq n\leq 6$, we can use iterations of these composition of maps to get the diagram
  \begin{tikzpicture}[node distance=5cm, auto]
  \node (A) {$\HH^*(B\homeo(S^{n})^\delta; \R)$};
  \node (B) [right of=A] {$\HH^*(B\homeo(\R^{n})^\delta; \R)$};
  \node (E) [right of=B, node distance=3.2cm] {$\dots$};
    \node (G) [right of=E, node distance=3.2cm] {$\HH^*(B\homeo(S^{3})^\delta; \R)$};
  \node (C) [below of=A, node distance=1.2cm] {$\hsb(B\homeo(S^{n})^\delta; \R)$};  
  \node (D) [below of=B, node distance=1.2cm] {$\hsb(B\homeo(\R^{n})^\delta; \R)$};
   \node (F) [below of=E, node distance=1.2cm] {$\dots$};
   \node (H) [below of=G, node distance=1.2cm] {$\hsb(B\homeo(S^{3})^\delta; \R).$};
  \draw[->] (C) to node {} (D);
    \draw[->] (D) to node {} (F);
  \draw[->] (B) to node {} (E);
  \draw [<-] (A) to node {}(C);
  \draw [->] (A) to node {} (B);
  \draw [<-] (B) to node {} (D);
  \draw [->] (E) to node {} (G);
  \draw [->] (H) to node {} (G);
  \draw [->] (F) to node {} (H);
\end{tikzpicture}

Recall that $\homeo(S^3)\simeq \mathrm{O}(4)$, so the Thurston--McDuff theorem implies that
\[\R[p_1, p_2]\to \HH^*(B\homeo(S^{3})^\delta; \R)\]
is an isomorphism. Therefore, the same argument as before implies that the classes $\R[p_1, p_2]\hookrightarrow \HH^*(B\homeo(S^{n})^\delta; \R)$ for $4\leq n\leq 6$ are unbounded.

If we continue the same diagram up $\HH^*(B\homeo(S^1)^\delta; \R)$, we obtain that the classes $\R[p_1]\hookrightarrow \HH^*(B\homeo(S^n)^\delta; \R)$ for $n=2,3$ are also unbounded.
\end{proof}

Let us conclude with a geometric interpretation of the unboundedness of characteristic classes in groups of homeomorphisms.

\begin{proposition}
\label{prop:unboundedness}
    Let $M$ be a connected topological manifold and let $G \leq \homeo(M)^\delta$. Let $n \geq 4$ and let $e \in \HH^n(BG; \R)$ be an unbounded class. Then there exist oriented closed connected $n$-manifolds $N_i$, maps $g_i \colon N_i \to BG$ such that the corresponding flat bundles $M \to E_i \to N_i$ satisfy
    \[\| N_i \| > 0 \qquad \text{and} \qquad \lim\limits_{i \to \infty} \frac{e(E_i \to N_i)}{\|N_i\|} = \infty.\]
\end{proposition}

\begin{proof}
Choose a cocycle $c \in \ZZ^n(BG; \R)$ representing $e$. For every $i \in \N$ there exists a singular $n$-cycle $z_i \in \ZZ_n(BG; \R)$ such that $|c(z_i)| > 2i \|z_i\|_1$.
By perturbing among cycles in the finite-dimensional space generated by the simplices appearing in $z_i$, we may assume that $z_i$ has rational coefficients and still satisfies $|c(z_i)| > i \|z_i\|_1$.
After multiplying by a positive integer and renaming, we may therefore assume that $z_i \in \ZZ_n(BG; \Z)$ is integral and $|c(z_i)| > i \|z_i\|_1$.
In particular, $z_i \neq 0$, so $\|z_i\|_1 \geq 1$, and hence
\[|c(z_i)| > i.\]
Realizing the finitely many simplices occurring in $z_i$ on a finite connected $\Delta$-complex $K_i$, we obtain an integral class $u_i \in \HH_n(K_i; \Z)$ and a map $h_i \colon K_i \to BG$ such that
\[ (h_i)_*(u_i) = [z_i] \quad \text{and} \quad \|u_i\|_1 \leq \|z_i\|_1. \]
Therefore
\[ |\langle h_i^* e,\, u_i \rangle| = |\langle e, [z_i] \rangle| = |c(z_i)| > i \|z_i\|_1 \geq i \|u_i\|_1. \]
Let $\Gamma_i \coloneqq \pi_1(K_i)$, let $c_i \colon K_i \to B\Gamma_i$ be the classifying map, and let $\rho_i \colon \Gamma_i \to G$ be the homomorphism induced by $h_i$. Since $BG$ is a $K(G,1)$, we have
\[ h_i \simeq B\rho_i \circ c_i. \]
Set
\[ \beta_i := (c_i)_*(u_i) \in \HH_n(B\Gamma_i; \Z) \subset \HH_n(B\Gamma_i; \R). \]
Because $c_i$ induces an isomorphism on $\pi_1$, the mapping theorem for the $\ell^1$-seminorm \cite[Corollary 5.2]{loehl1} gives
\[ \|\beta_i\|_1 = \|u_i\|_1. \]
Hence
\[ |\langle B\rho_i^* e,\, \beta_i \rangle| = |\langle h_i^* e,\, u_i \rangle| = |c(z_i)| > \max\{i, i \|\beta_i\|_1\}.\]
Since $K_i$ is finite, the group $\Gamma_i$ is finitely presented. By the normed Thom realization theorem \cite[Theorem 8.1]{spectrumsv} (this is where we use the assumption $n \geq 4$), there exists a dimensional constant $C_n \in \N$, a number $m \in \{1, \ldots, C_n\}$, an oriented closed connected $n$-manifold $N_i$ and a map $f_i \colon N_i \to B\Gamma_i$ such that
\[ (f_i)_*[N_i] = m \cdot \beta_i \quad \text{and} \quad \|N_i\| = m \cdot \|\beta_i\|_1. \]

Set
\[ g_i \coloneqq B\rho_i \circ f_i \colon N_i \to BG, \]
and let $M \to E_i \to N_i$ be the corresponding flat $M$-bundle. Then
\[ |e(E_i \to N_i)| = |\langle g_i^* e, [N_i] \rangle| = |\langle B\rho_i^* e, \beta_i \rangle| > i \| \beta_i\|_1 = \frac{i}{m} \|N_i\| \geq \frac{i}{C_n} \|N_i \| \]
and also $|e(E_i \to N_i)| > i$.
So whenever $\|N_i\| > 0$, we have
\[ \frac{|e(E_i \to N_i)|}{\|N_i\|} > \frac{i}{C_n}.\]
If this happens for infinitely many $i$, we may pass to that subsequence and reindex, then the desired limit follows immediately.

It remains to treat the case that $\|N_i\| = 0$ for all sufficiently large $i$. Fix an oriented closed connected $n$-manifold $Q$ with $\|Q\| > 0$; for example one may take $Q$ to be a closed hyperbolic $n$-manifold \cite[Section 8.13]{frigerio}. For all large $i$, let
\[ N_i^\sharp \coloneqq N_i \# Q, \]
and let $\tilde{g}_i \colon N_i^\sharp \to BG$ be a map that restricts to $g_i$ on $N_i$ and is constant on $Q$ (collapse the connected-sum neck to the basepoint of $BG$). Then
\[ |\langle \tilde{g}_i^* e, [N_i^\sharp] \rangle| = |\langle g_i^* e, [N_i] \rangle| = |e(E_i \to N_i)| > i, \]
while additivity of simplicial volume \cite[Corollary 7.7]{frigerio} gives
\[ \|N_i^\sharp\| = \|N_i\| + \|Q\| = \|Q\| > 0. \]
Hence
\[ \frac{|\langle \tilde{g}_i^* e, [N_i^\sharp] \rangle|}{\|N_i^\sharp\|} > \frac{i}{\|Q\|} \xrightarrow{i \to \infty} \infty. \]
Renaming $N_i^\sharp$ as $N_i$ finishes the proof.
\end{proof}
\begin{remark}
    An interesting refinement of the proposition would be to fix the base of bundles and just vary the bundles over the same base. The authors do not know if this refinement is true; this is related to \cite[Question 7.7]{MN}. 
\end{remark}
\section{Further results and questions}
\label{sec:further}

We conclude by mentioning other settings in which Theorem \ref{thm:criterion} applies, with minor variations of the arguments of Section \ref{sec:verifying}. On the other hand, when trying to tackle other transformation groups of interest, we encounter technical difficulties, and we leave these as open questions. Let us stress that even though Theorem \ref{thm:criterion} may not apply for some of these cases, the general proof strategy seems fruitful beyond its specific scope. For example, our proof shows that the $\mathbf{W}$ property is nothing but a simple combinatorial condition to ensure bounded acyclicity of posets of sequences, and more sophisticated criteria could work in more general contexts. Moreover, our statement of Theorem \ref{thm:resolutions} made the convenient assumption that the stabilizers belong to finitely many isomorphism classes, which was enough for our purposes, but the principle works more generally, as long as the \emph{vanishing modulus} of the stabilizers is controlled \cite[Theorem 3.3]{MN}.

\subsection{Results}

We start by considering other transformation groups of the line. Let $\PLpw(\R)$ denote the group of orientation-preserving piecewise linear homeomorphisms of $\R$ with a closed discrete set of breakpoints. Subgroups of this group are central in the theory of left-orderable groups, and include lifts of piecewise linear groups of the circle \cite{ghyssergiescu}, as well as the first examples of finitely generated simple left-orderable groups \cite{hydelodha, mbt}. The bounded acyclicity of these simple groups was already proven in \cite[Example 5.41]{ccc}; the proof relies on their \emph{quasi-periodic} nature, and in particular it does not generalize to the whole of $\PLpw(\R)$; more recently finitely presented examples were constructed \cite{hydelodha2}, however they admit a proximal action on the circle, hence they have non-vanishing second bounded cohomology \cite{ghys}.  With a very minor variation on the arguments in Section \ref{verifying:line}, we obtain:

\begin{theorem}
\label{thm:PL}
    The group $\PLpw(\R)$ is boundedly acyclic.
\end{theorem}

\begin{proof}
    Any subgroup of $\homeoc(\R)$ acting without global fixed points has commuting $\Z$-conjugates \cite[Corollary 5.32]{ccc}. Therefore, by the same argument as in the proof of Theorem \ref{thm:main:homeo} for $\homeop(\R)$, it suffices to prove bounded acyclicity of the subgroup $G < \PLpw(\R)$ of elements that fix pointwise $(-\infty, 0]$. We consider the poset $\X$ of strictly increasing diverging sequences of \emph{points} in $\R_{\geq 0}$. The action of $G$ is transitive, and this poset satisfies the $\mathbf{W}$ property, by Proposition \ref{prop:abstract:W} (where $x \simeq y \Leftrightarrow x = y$ and $x \sqsubseteq y \Leftrightarrow x \leq y$). The stabilizer of a sequence is isomorphic to $\PLp(I)^{\N}$, where $I = [0, 1]$ and $\PLp(I)$ denotes the group of piecewise linear homeomorphisms of $I$ with a \emph{finite} set of breakpoints. Because germs at fixed points are abelian, $[\PLp(I), \PLp(I)]^{\N}$ is a product of subgroups of $\homeoc(\R)$ acting without global fixed points, so it has commuting $\Z$-conjugates by \cite[Corollary 5.32]{ccc} and \cite[Lemma 4.8]{ccc}. It then follows from Theorem \ref{thm:ccc} and coamenability \cite[Proposition 8.6.6]{monod} that $\PLp(I)^{\N}$ is boundedly acyclic, being an extension of a group with commuting $\Z$-conjugates by an abelian group. Thus Theorem \ref{thm:criterion} applies, and we conclude.
\end{proof}

The proof generalizes seamlessly to the piecewise projective case $\PPpw(\R)$; again, subgroups of this group have provided very important examples in group theory, particularly in the realm of non-amenable groups without free subgroups \cite{monod:PP, lodhamoore}:

\begin{theorem}
    The group $\PPpw(\R)$ is boundedly acyclic.
\end{theorem}

\begin{proof}
    The proof of Theorem \ref{thm:PL} applies, with the only modification that germs at fixed points are not abelian but metabelian, and so instead of the product of commutator subgroups, we consider the product of double commutator subgroups.
\end{proof}

Next, we consider point stabilizers in transformation groups of Euclidean spaces. Let $\diffrp(\R^n, 0)$ denote the stabilizer of $0$ in $\diffrp(\R^n)$. For $r \geq 1$, the derivative at $0$ gives a surjective homomorphism $\diffrp(\R^n, 0) \to \mathrm{GL}_n(\R)_+$, which admits a section: the inclusion $\mathrm{GL}_n(\R)_+ \to \diffrp(\R^n, 0)$. This implies that the bounded cohomology of $\mathrm{GL}_n(\R)_+$ embeds into the bounded cohomology of $\diffrp(\R^n)$, which is therefore not boundedly acyclic \cite{MR0686042}. On the other hand, we prove:

\begin{theorem}
\label{thm:homeo0}
    For all $n \in \N^*$, the group $\homeop(\R^n, 0)$ is boundedly acyclic.
\end{theorem}

\begin{proof}
    We identify $\homeop(\R^n, 0) \cong \homeop(S^n, \{ 0, \infty \})$, which gives a short exact sequence
    \[1 \to \homeo(S^n, \mathrm{near } \{0, \infty\}) \to \homeop(S^n, \{ 0, \infty \}) \to \mathcal{G}_0 \times \mathcal{G}_\infty \to 1.\]
    Here $\homeo(S^n, \mathrm{near } \{0, \infty\})$ denotes the group of homeomorphisms of $S^n$ that fix pointwise a neighborhood of $\{ 0, \infty \}$, and $\mathcal{G}_0$ and $\mathcal{G}_\infty$ denote the germs at $0$ and $\infty$, respectively. By Theorem \ref{thm:compactly} the group $\homeo(S^n, \mathrm{near }\{ 0, \infty \}) \cong \homeoc(S^{n-1} \times \R)$ is boundedly acyclic, so Proposition \ref{prop:quotient} reduces the statement to the bounded acyclicity of $\mathcal{G}_0 \times \mathcal{G}_\infty$. Since $\mathcal{G}_0 \cong \mathcal{G}_\infty$, it suffices to show that $\mathcal{G}_\infty$ is boundedly acyclic. We also have a short exact sequence
    \[1 \to \homeoc(\R^n) \to \homeop(\R^n) \to \mathcal{G}_\infty \to 1,\]
    so Theorem \ref{thm:compactly}, Theorem \ref{thm:main:homeo} and Proposition \ref{prop:quotient} imply that $\mathcal{G}_\infty$ is boundedly acyclic.
\end{proof}

In a similar way, we can deal with compact annuli, which were mentioned in the introduction. We denote by $\homeoo(S^n \times [0, 1], \partial)$ the identity component of the group of homeomorphisms of the annulus fixing both boundary components pointwise.

\begin{proof}[Proof of Theorem \ref{thm:annuli:homeo}]
    Consider the short exact sequence
    \[1 \to \homeoco(S^n \times (0, 1)) \to \homeoo(S^n \times [0, 1], \partial) \to \mathcal{G}_0 \times \mathcal{G}_1 \to 1.\]
    Here $\mathcal{G}_i$ denotes the groups of germs of homeomorphisms at the boundary component $S^n \times \{ i \}$. As in the proof of Theorem \ref{thm:homeo0}, it suffices to show that $\mathcal{G}_0 \cong \mathcal{G}_1$ is boundedly acyclic. This follows from Theorem \ref{thm:discs}, by considering the short exact sequence
    \[1 \to \homeoc(B^{n+1}) \to \homeo(D^n, \partial) \to \mathcal{G}_1 \to 1. \qedhere\]
\end{proof}

\begin{corollary}
\label{cor:annulus:restriction}
    For all $n \in \N^*$, the restriction $\homeoo(S^n \times [0, 1]) \to \homeoo(S^n) \times \homeoo(S^n)$ induces an isomorphism in bounded cohomology in all degrees.
\end{corollary}

\begin{proof}
    Consider the short exact sequence
    \[1 \to \homeoo(S^n \times [0, 1], \partial) \to \homeoo(S^n \times [0, 1]) \to \homeoo(S^n) \times \homeoo(S^n) \to 1\]
    and apply Theorem \ref{thm:annuli:homeo} and Proposition \ref{prop:quotient}.
\end{proof}

As in Corollary \ref{cor:disc:low}, using \cite[Theorems 1.8 and 1.10]{MN} we can now compute the bounded cohomology of $\homeoo(S^n \times [0, 1])$ for $n \geq 2$, in low degrees (note that in this case we also need a degree-wise version of Proposition \ref{prop:quotient}, see \cite[Proposition 2.4]{MN}).

\begin{corollary}
\label{cor:annulus:low}
    For all $n \geq 2$, the bounded cohomology of $\homeoo(S^n \times [0, 1])$ vanishes in degrees $2$ and $3$. The bounded cohomology of $\homeoo(S^3 \times [0, 1])$ vanishes in degree $4$.
\end{corollary}

In case $n = 1$ we can say more: Corollary \ref{cor:annuli:regularity} describes the entire bounded cohomology ring, also in the case of diffeomorphisms.

\begin{proof}[Proof of Corollary \ref{cor:annuli:regularity}]
    The first statement was already proved in Corollary \ref{cor:annulus:restriction} for $r = 0$. For $r \geq 1$, it follows from a similar argument, given the bounded acyclicity of $\diffro(D^2, \partial)$, which is part of the proof that the restriction $\diffro(D^2) \to \diffro(S^1)$ induces an isomorphism in bounded cohomology \cite[Theorem 1.4]{MN}.
    
    For the second statement, \cite[Theorem 1.2]{MN} gives an isomorphism of graded $\R$-algebras $\hsb(\diffro(S^1)) \cong \R[\mathscr{E}]$. The computation for the product is identical to the one for Thompson's group $T$ in \cite[Proposition 6.13]{binate}. More precisely: the natural map $\R[\mathscr{E}_0, \mathscr{E}_1] \to \hsb(\diffro(S^1) \times \diffro(S^1))$ is injective, since the composition with the comparison map is. The Lyndon--Hochschild--Serre spectral sequence in bounded cohomology \cite[Section 12]{monod} applies because the canonical seminorm on $\hsb(\diffro(S^1))$ is a norm, and this gives an upper bound on the dimension of $\hsb(\diffro(S^1) \times \diffro(S^1))$, proving surjectivity.
\end{proof}

Using the computation in degree $2$, we recover a theorem of Militon. For a group $G$, we denote by $Q(G)$ the space of homogeneous quasimorphism $G \to \R$, and by $\HH^1(G)$ the subspace of homomorphisms. There is a natural homogeneous quasimorphism on $\diffro(S^1 \times [0, 1])$, called the \emph{torsion number} \cite{militon}. Given $f \in \diffro(S^1 \times [0, 1])$, let $F \in \diffro(\R \times [0, 1])$ be a lift to the universal cover. The restriction to the two boundary components defines maps $F_0, F_1$ in the universal cover of $\homeoo(S^1)$, to which we can associate the asymptotic translation number $\tau$.

\begin{definition}
\label{def:torsionnumber}
    The \emph{torsion number} is defined as $\rho(f) \coloneqq \tau(F_0) - \tau(F_1)$.
\end{definition}

This is well-defined: if we replace $F$ by another lift, then $\tau(F_0)$ and $\tau(F_1)$ are both shifted by an integer, so their difference is the same. Moreover, $\rho$ is a homogeneous quasimorphisms, in fact, $\delta \rho$ is the difference of the canonical real bounded Euler cocycles of $\mathscr{E}_1$ and $\mathscr{E}_0$, as follows from the definitions \cite{matsumoto} (see also \cite[Section 10.8]{frigerio}).

\begin{proof}[Proof of Corollary \ref{cor:militon}]
    There is an exact sequence (see e.g. \cite[Proposition 2.8]{frigerio}):
    \[0 \to \HH^1(G) \to Q(G) \to \htb(G) \to \HH^2(G).\]
    We know that $\htb(G) = \R \mathscr{E}_0 \oplus \R \mathscr{E}_1$ by Corollary \ref{cor:annuli:regularity}. Both $\mathscr{E}_0$ and $\mathscr{E}_1$ are non-trivial in $\HH^2(G)$; this is because the restriction $G \to \diffro(S^1)$ to either boundary component admits a section. The fact that $\delta \rho$ is the difference of the canonical real bounded Euler cocycles of $\mathscr{E}_1$ and $\mathscr{E}_0$ implies at once that $\mathscr{E}_1 - \mathscr{E}_0$ generates the kernel of $\htb(G) \to \HH^2(G)$, and that $Q(G)/\HH^1(G)$ is one-dimensional, spanned by $\rho$.

    The last statement is a consequence of the fact that, for $r \neq 2, 3$, the group $G$ is perfect \cite[Theorem 5]{rybicki} (see also \cite{abefukui}).
\end{proof}

\subsection{Questions}

A very obvious question is whether Theorem \ref{thm:discs} holds in higher regularity:

\begin{question}
    Is $\diffr(D^n, \partial)$ boundedly acyclic for $r \geq 1$ and $n \geq 3$?
\end{question}

The case of $n  = 2$ is known \cite[Theorems 1.3 and 1.4]{MN}, by methods specific to dimension $2$. Our method is not dependent on dimension, but dealing with differentiability at the boundary presents several complications.

The main application of Theorem \ref{thm:discs} is the isomorphism in bounded cohomology between homeomorphism groups of discs and spheres (Corollary \ref{cor:discs}). However, the bounded cohomology of $\homeop(S^n)$, and more generally of $\diffrp(S^n)$, is understood (beyond low degrees) only in dimension $1$ \cite[Theorems 1.1 and 1.2]{MN}.

\begin{question}
    Is $\diffrp(S^n)$ boundedly acyclic for $r \geq 0$ and $n \geq 2$?
\end{question}

This was recently answered positively for $n \geq 4$ and $r \geq 1$ \cite{spheres}. See also \cite[Questions 7.3 and 7.4]{MN} for related questions.

\medskip

One interesting feature of the proof of bounded acyclicity of $\homeo(\R^n)$ was that we had to study the bounded cohomology of certain monoids of self-embeddings. Normally the classifying space of the monoid of self-embeddings is more flexible to study than that of their maximal subgroup. But we had to go through auxiliary groups to prove bounded acyclicity of $\emb(\N^*)$. On the other hand $\homeo(\R^n)$ is the maximal subgroup of the monoid $\mathrm{Emb}(\R^n)$ of self-embeddings of $\R^n$. As the topological monoid, Kister \cite{MR0180986} showed that the inclusion
\[
\homeo(\R^n)\to \mathrm{Emb}(\R^n),
\]
induces a weak homotopy equivalence. As a discrete monoid, Segal's results in \cite{segal1978classifying} imply that 
\[
B\homeo(\R^n)^\delta\to B\mathrm{Emb}^{\delta}(\R^n),
\]
induces a homology isomorphism. Now that our main theorem proves the bounded acyclicity of $\homeo(\R^n)$ as a discrete group, it seems natural to ask the following question.
\begin{question}
    Is $\mathrm{Emb}^{\delta}(\R^n)$ boundedly acyclic?
\end{question}

Hirsch--Thurston \cite{MR0370615} proved that the Euler class for flat $C^0$-sphere bundles over an amenable base vanishes. Moreover, Calegari \cite[Theorem D]{Calegari} showed that the Euler class for $C^1$-flat $\R^2$-bundles over a torus should vanish. Even though the Euler class vanishes for such bundles with amenable bases, we know that the Euler class is unbounded  \cite{MN, Calegari}. This suggests the following question.
\begin{question}
    Are the Euler classes in $\HH^{2n}(\diffr(\R^{2n}, \mathrm{vol});\R)$ and in $\HH^{2n}(\diffr(S^{2n-1}, \mathrm{vol});\R)$ bounded?
\end{question}
Let us point out that $\diffrc(\R^n; \mathrm{vol})$ is boundedly acyclic \cite[Corollary 1.12]{ccc}, but our methods from Section \ref{verifying:Rn} break down in several ways when trying to prove that $\diffr(\R^n; \mathrm{vol})$ is boundedly acyclic.

\appendix
\section{Controlled annulus theorem (by Alexander Kupers)}
\label{appendix}

In this appendix, we will prove that the action of $\homeo(D^n,\partial)$ of homeomorphisms of the closed disc fixing the boundary pointwise on the set $\mathcal{F}$ of ``fat sequences'' is transitive: combine Theorem \ref{thm:homeo-dd} with Corollary \ref{cor:add-germs}.

\subsection{Sequences of spheres in \texorpdfstring{$\mathbb{R}^n$}{}} As preparation, we will study the action of homeomorphisms of $\mathbb{R}^n$ on sequences of spheres. A \emph{parametrised sphere} in $\R^n$ is a locally flat embedding $S \colon S^{n-1} \hookrightarrow \mathbb{R}^n$ so that the composition 
\[S^{n-1} \overset{S}\lra \mathbb{R}^n \backslash\{\text{point in bounded component of $\mathbb{R}^n\backslash \mathrm{im}(S)$}\} \overset{\pi}\lra S^{n-1}\]
is orientation-preserving (i.e.~preserves the fundamental class). We say $S_1$ \emph{englobes} $S_0$ if $S_0$ is contained in the bounded component of $\mathbb{R}^n \setminus S_1$.

\begin{definition}A \emph{sequence} of parametrised spheres is a collection $\mathbf{S} = \{S_i\}_{i \in \mathbf{N}^\ast}$ of disjoint parametrised spheres such that $S_i$ englobes $S_{i-1}$.\end{definition}

\begin{example}The standard sequence $\mathbf{\Sigma}$ of parametrised spheres is given by letting $\Sigma_i \colon S^{n-1} \hookrightarrow \mathbb{R}^n$ be given by $x \mapsto i\cdot x$.\end{example}

The identity component of the group of homeomorphisms of $\mathbb{R}^n$ acts on the collection $\mathcal{S}$ of all sequences of parametrised spheres by post-composition. As a warm-up we will prove the following:

\begin{theorem}\label{thm:homeo-rd} A sequence $\mathbf{S}$ lies in the orbit of the standard sequence $\mathbf{\Sigma}$ under $\homeo_\circ(\mathbb{R}^n)$ if and only if every compact subset of $\mathbb{R}^n$ intersects finitely many $S_i$.
\end{theorem}

We refer to the condition in the previous definition as $\mathbf{S}$ \emph{going to infinity}. It is equivalent to $\mathbb{R}^n$ being equal to the union of the bounded components of $\mathbb{R}^n \setminus \mathrm{im}(S_i)$ for all $i \geq 1$: if a compact $K$ intersects the image of infinitely many $S_i$, pick a sequence $k_j \in K$ with $k_j \in \mathrm{im}(S_{i_j})$ where $i_j \to \infty$ and let $k$ be the limit of a convergent subsequence, then we must have that $k$ must lie in the unbounded component of $\mathbb{R}^n \setminus \mathrm{im}(S_{i_j})$ for all $j$ and hence the intersection of the unbounded components of all $\mathbb{R} \setminus \mathrm{im}(S_i)$ is non-empty.

The proof will use as main input two slightly sharpened variants of the Schoenflies and annulus theorems.

\begin{theorem}[Schoenflies theorem] \label{thm:schoenflies} Let $S$ be a parametrised sphere and $W$ denote the closure of the bounded component of $\mathbb{R}^n \setminus \mathrm{im}(S)$. Then there is a homeomorphism $h \colon D^n \to W$ that agrees with $S$ on $S^{n-1}$.\end{theorem}

\begin{proof}The usual Schoenflies theorem provides a homeomorphism $g \colon D^n \to W$ (\cite{BrownSchoenflies,Mazur,Morse}). By the Alexander trick $(g|_{S^{n-1}})^{-1} \circ S$ extends to a homeomorphism of $D^n$ and precompose $g$ with this to get the desired homeomorphism.
\end{proof}

\begin{theorem}[Annulus theorem] \label{thm:annulus} Let $S_1,S_2$ be parametrised spheres with $S_1$ englobed by $S_2$, and let $W$ denote the closure of the region between these. Then there is a homeomorphism $h \colon S^{n-1} \times [1,2] \to W$ that agrees with $S_1$ on $S^{n-1} \times \{1\}$ and $S_2$ on $S^{n-1} \times \{2\}$.\end{theorem}

\begin{proof}The usual annulus theorem provides a homeomorphism $h\colon S^{n-1} \times [1,2] \to W$ \cite{Moise,Kirby,Quinn}. We may assume by precomposing with a reflection if necessary that $h$ is orientation-preserving on both boundary components. Then $(h|_{S^{n-1} \times \{1\}})^{-1} \circ S_1$ is orientation-preserving and since $\homeop(S^{n-1})$ is path-connected, there is an isotopy $b_t \colon S^{n-1} \times [0,1] \to S^{n-1}$ from it to the identity. Similarly, there is an isotopy $b'_t$ from the identity to $(h|_{S^{n-1} \times \{2\}})^{-1} \circ S_2$. Now define a homeomorphism of $S^{n-1} \times [1,2]$ by
	\[(x,t) \longmapsto \begin{cases} (b_{2t}(x),t) & \text{if $t \in [0,1/2]$,} \\
		(b'_{2t-1}(x),t) & \text{if $t \in [1/2,1]$} \end{cases}\]
and precompose $h$ with it to get the desired homeomorphism. 
\end{proof}

\begin{proof}[Proof of Theorem \ref{thm:homeo-rd}] The direction $\Rightarrow$ is clear as the standard sequence has the property of going to infinity and this property is preserved by homeomorphisms. For the direction $\Leftarrow$, without any hypothesis on $\mathbf{S} = \{S_i\}_{i \in \mathbf{N}^*}$ we will inductively construct an embedding $h \colon \mathbb{R}^n \to \mathbb{R}^n$ so that $h \Sigma_i = S_i$. The condition that the sequence goes to infinity guarantees this is surjective, and any surjective embedding is a homeomorphism. It will lie in the identity component because by construction it will be orientation-preserving.
	
Let $W_1$ denote the closure of the bounded component of $\mathbb{R}^n \setminus \mathrm{im}(S_1)$ and $W_{i+1} \colon S_i \leadsto S_{i+1}$ for $i \geq 1$ denote the closure of the region between $\mathrm{im}(S_i)$ and $\mathrm{im}(S_{i+1})$. Apply Theorem \ref{thm:schoenflies} to get a homeomorphism $h_1 \colon D^n \to W_1$ and apply Theorem \ref{thm:annulus} to get homeomorphisms $h_{i+1} \colon S^{n-1} \times [i,i+1] \to W_{i+1}$, all agreeing with the appropriate $S_i$'s on their boundary components. Now we define $h$ as
\[h(x) \coloneqq \begin{cases} h_1(x) & \text{if $x \in D^n$,} \\
		h_{i+1}(x) & \text{if $x \in S^{n-1} \times [i,i+1]$ for $i \geq 1$,}\end{cases}\]
where we identify $S^{n-1} \times [i,i+1]$ with a subset of $\mathbb{R}^n$ using radial coordinates.
\end{proof}

\subsection{Sequences of spheres in \texorpdfstring{$D^n$}{}}
Let $\homeo(D^n,\partial)$ denote those homeomorphisms that are the identity on $\partial D^n$ (\emph{not} necessarily near the boundary). Identifying the interior of $D^n$ with $\mathbb{R}^n$, these are those homeomorphisms $h$ of $\mathbb{R}^n$ such the function $\tfrac{h(tx)}{|h(tx)|} \colon S^{n-1} \to S^{n-1}$, defined for $t$ sufficiently large, converges uniformly to the identity at $t \to \infty$. We may ask what the orbit of the standard sequence under $\homeo(D^n, \partial) \subset \homeop(\mathbb{R}^n)$ is. 

To state the answer, we introduce a notion of ``radial control'' using reference maps to $S^{n-1}$, with metric $d$ induced from the standard one on Euclidean space. Given reference maps $r_0 \colon X_0 \to S^{n-1}$ and $r_1 \colon X_1 \to S^{n-1}$, we say that a continuous map $f \colon X_0 \to X_1$ has \emph{diameter $<\epsilon$} if $\sup\{d(r_0(x),r_1(f(x)))\mid x \in X_0\}<\epsilon$. The map $\smash{x \mapsto \tfrac{x}{|x|}}$ induces a map to $S^{n-1}$ on any subset of $\mathbb{R}^n$ avoiding the origin; unless we say otherwise, this is the reference map to $S^{n-1}$ on such subsets. Let us make some elementary observations about this notion:
\begin{itemize}
	\item If $f,g$ are composable so that $\diam(f)<\epsilon$ and $\diam(g)<\delta$ then $\diam(gf)<\epsilon+\delta$.
	\item If $h$ is a homeomorphism so that $\diam(h)<\epsilon$ then $\diam(h^{-1})<\epsilon$.
	\item If $\{A_i\}$ is a cover then $\diam(f) = \sup\,\diam(f|_{A_i})$.
\end{itemize}

\begin{theorem}\label{thm:homeo-dd} A sequence of parametrised spheres $\mathbf{S} = \{S_i\}_{i \in \mathbb{N}^*} \in \mathcal{S}$ lies in the orbit of the standard  sequence $\mathbf{\Sigma}$ under $\homeo(D^n,\partial)$ if and only if $\diam(S_i) \to 0$ as $i \to \infty$.
\end{theorem}
We refer to the condition in the theorem as $\mathbf{S}$ having \emph{trivial germ at the boundary}.

This theorem will follow by essentially the same proof as for Theorem \ref{thm:homeo-rd} once we establish the following ``controlled annulus theorem''. Such a result was surely known to the experts, but we could not find it in the literature.

\begin{proposition}[Controlled annulus theorem] \label{prop:controlled-annulus}  Let $S_1,S_2$ be parametrised spheres with $S_1$ englobing the origin and $S_2$ englobing $S_1$, and let $W$ denote the closure of the region between these. For all $\epsilon>0$ there exists a $\delta>0$ so that if $S_1,S_2$ have diameter $<\delta$, then there is a homeomorphism $h \colon S^{n-1} \times [1,2] \to W$ of diameter $<\epsilon$ that agrees with $S_1$ on $S^{n-1} \times \{1\}$ and $S_2$ on $S^{n-2} \times \{2\}$.
\end{proposition}

The proof relies on local contractibility of spaces of locally flat embeddings, in the compact-open topology. If $Y$ is a metric space and $X$ any space, for any two maps $f,g \colon X \to Y$ we define $d(f,g) = \sup\{d(f(x),g(x)) \mid x \in X\}$. The following is a consequence of \cite{EdwardsKirby}, explicitly stated as \cite[Theorem 7.3.1]{DavermanVenema}.

\begin{theorem} \label{thm:local-contractibility} Let $N$ be a compact PL-manifold with metric and $M \subset \mathrm{int}\,N$ be a two-sided closed PL-submanifold of codimension one. For all $\epsilon>0$ there exists a $\delta>0$ so that for any locally flat embedding $e \colon M \to \mathrm{int}\,N$ satisfying $d(e,\mathrm{inc}_M)<\delta$ there is a homeomorphism $h \colon N \to N$ fixing the boundary pointwise, satisfying $he = \mathrm{inc}_M$, and $d(h,\id_N)<\epsilon$.	
\end{theorem}

\begin{proof}[Proof of Proposition \ref{prop:controlled-annulus}] We suppose $S_1,S_2$ have diameter $<\delta$ for some small $\delta$ to be specified at the end of the proof.

Picking spheres of radii $R_{\min}$ and $R_{\max}$ such that $S_1$ lies between them, using radial coordinates we can regard $S_1$ as an embedding $S^{n-1} \to S^{n-1} \times [R_{\min},R_{\max}]$. By picking a suitable self-homeomorphism $\lambda$ of $[R_{\min},R_{\max}]$ fixing the boundary, we can assume $d((\id \times \lambda) S_1,\mathrm{inc}_{S^{n-1} \times \{R_\mathrm{ave}\}})<\delta$ where $R_\mathrm{ave} \coloneqq \smash{\tfrac{1}{2}}(R_{\min}+R_{\max})$. For $\epsilon'$ to be picked later, by taking $\delta$ to be small we can use Theorem \ref{thm:local-contractibility} to find a homeomorphism $\widetilde{g}_1$ of $S^{n-1} \times [R_{\min},R_{\max}]$ that fixes the boundary, satisfying $d(\widetilde{g}_1,\id_{S^{n-1} \times [R_{\min},R_{\max}]})<\epsilon'$ and $\widetilde{g}_1 (\id \times \lambda) S_1=\mathrm{inc}_{S^{n-1} \times \{R_{\mathrm{ave}}\}}$. Extending $\widetilde{g}_1 (\id \times \lambda)$ by the identity and rescaling the radial coordinates by $\tfrac{1}{R_\mathrm{ave}}$, we get a homeomorphism $g_1$ of diameter $<\epsilon'$ so that $g_1S_1 = \mathrm{inc}_{S^{n-1} \times \{1\}}$. 

Now pick a sphere of radius $R'_{\max}$ so that $g_1S_2$ lies between it and of that of radius $1$. Using that $\diam(g_1S_2)<\epsilon'+\delta$, by making $\epsilon'$ and $\delta$ smaller we can use the same argument to construct a homeomorphism $g_2$ supported in $S^{n-1} \times [1,R'_{\max}]$ of diameter $<\epsilon''$ for $\epsilon' \leq \epsilon''$ and satisfying $g_2g_1S_2 = \mathrm{inc}_{S^{n-1} \times \{2\}}$. Then $(g_2g_1)^{-1}$ restricts to a homeomorphism $S^{n-1} \times [1,2] \to W$ of diameter $<\epsilon'+\epsilon''$ agreeing with $S_i$ on $S^{n-1} \times \{i\}$ for $i=1,2$. By taking $\delta$ small enough we can make $\epsilon'+\epsilon'' \leq \epsilon$.
\end{proof}

\begin{proof}[Proof of Theorem \ref{thm:homeo-dd}] The direction $\Leftarrow$ is clear as the standard sequence satisfies the condition of having trivial germ at the boundary and this condition is preserved by homeomorphisms of $D^n$ fixing $\partial D^n$ pointwise. For the direction $\Rightarrow$, we will construct an embedding $h \colon \mathbb{R}^n \to \mathbb{R}^n$ so that $h \Sigma_i = S_i$ and $\diam(h|_{S^{n-1} \times t}) \to 0$ as $t \to \infty$ using the second condition on the nested sequence. As before, the first condition guarantees this is surjective so a homeomorphism and it will be orientation-preserving so in the identity component.
	
Let $W_1$ denote the closure of the bounded component of $\mathbb{R}^n \setminus \mathrm{im}(S_1)$ and $W_{i+1} \colon S_i \leadsto S_{i+1}$ denote the closure of the region between $\mathrm{im}(S_i)$ and $\mathrm{im}(S_{i+1})$. Pick now for each $k \geq 1$ a $\delta_k>0$ so that Proposition \ref{prop:controlled-annulus} applies with $\epsilon = 1/k$.

We first apply the Schoenflies theorem (Theorem \ref{thm:schoenflies}) to get a homeomorphism $h_1 \colon D^n \to W_1$ and then would like to apply the controlled annulus theorem (Proposition \ref{prop:controlled-annulus}) to get a homeomorphism $h_{i+1} \colon S^{n-1} \times [i,i+1] \to W_{i+1}$ of diameter $<1/k$ where $k$ is the largest integer such that $\diam(S_i),\diam(S_{i+1})<\delta_k$. There will be finitely many $i$ where this is not possible, and in these cases we instead apply the annulus theorem (Theorem \ref{thm:annulus}) to get homeomorphisms $h_{i+1} \colon S^{n-1} \times [i,i+1] \to W_{i+1}$. All of these agree with the appropriate $S_i$'s on the boundary. Now we can define $h$ as
	\[h(x) \coloneqq \begin{cases} h_1(x) & \text{if $x \in D^n$,} \\
		h_{i+1}(x) & \text{if $x \in S^{n-1} \times [i,i+1]$ for $i \geq 1$,}\end{cases}\]
where we identify $S^{n-1} \times [i,i+1]$ with a subset of $\mathbb{R}^n$ using radial coordinates, and by construction $\mathrm{diam}(h|_{S^{n-1} \times \{t\}}) \to 0$ as $t \to \infty$.
\end{proof}

\subsection{Sequences of germs in \texorpdfstring{$\mathbb{R}^n$}{} or \texorpdfstring{$D^n$}{}} We can also consider germs near $S^{n-1} \times \{0\}$ of orientation-preserving topological embeddings $\overline{S} \colon S^{n-1} \times (-\tfrac{1}{4},\tfrac{1}{4}) \hookrightarrow \mathbb{R}^n$. 

\begin{definition}A \emph{sequence of germs of parametrised spheres} is a collection $\overline{\mathbf{S}} = \{\overline{S}_i\}_{i \geq 1}$ of disjoint germs of parametrised spheres such that $\overline{S}_i$ englobes $\overline{S}_{i-1}$.\end{definition}

There is a standard such sequence $\overline{\mathbf{\Sigma}}$ given by the germs of $\overline{\Sigma}_i \colon S^{n-1} \times (-\tfrac{1}{4},\tfrac{1}{4}) \to \mathbb{R}^n$ given by $(x,s) \mapsto (i+s) \cdot x$. We will use the notation $\mathring{S}_i \coloneqq \overline{S}_i|_{S^{n-1} \times \{0\}}$ for the cores and $\mathring{\mathbf{S}}$ for the resulting sequence of parametrised spheres.

\begin{corollary}\label{cor:add-germs} A sequence $\overline{\mathbf{S}}$ of germs of parametrised spheres lies in the orbit of the standard sequence $\overline{\mathbf{\Sigma}}$ under $\homeo^+(\mathbb{R}^n)$ or $\homeo(D^n,\partial)$ if and only if its sequence of cores $\mathring{\mathbf{S}}$ does.
\end{corollary}

\begin{proof}
	The direction $\Rightarrow$ is obvious. For $\Leftarrow$, given $\overline{\mathbf{S}} = \{\overline{S}_i\}_{i \in \mathbf{N}^\ast}$ we pick a homeomorphism sending $\mathring{\mathbf{S}}$ to the standard sequence. It remains to adjust the germs near the standard spheres and to do so, we will use that for any germ of embedding $S^{n-1} \times [0,1) \to S^{n-1} \times [0,1)$ near $S^{n-1} \times \{0\}$ that is the identity on $S^{n-1} \times \{0\}$, there is a compactly-supported homeomorphism that is the identity on the boundary which sends it to the germ of the identity of arbitrarily small diameter. To find this, note that both the given germ and the identity germ can be represented by collars on the boundary and invoke the collaring uniqueness theorem \cite[Theorem A.1, Essay I]{KirbySiebenmann}, which comes with control. We now pick for each germ $\overline{S}_i$ extending $\Sigma_i$ two homeomorphisms as above---one for each side, in $S^{n-1} \times (i-1/4,i]$ and $S^{n-1} \times [i,i+1/4)$---that map its germ to the standard germ, have diameter smaller than $1/i$, and necessarily agree on $S^{n-1} \times \{i\}$ so may be combined into a single homeomorphism. The resulting homeomorphisms, one for each $i \in \mathbf{N}^\ast$, by construction have disjoint compact support so may be combined into a single homeomorphism $h$ so that $h \overline{S}_i = \overline{\Sigma}_i$ and $\mathrm{diam}(h|_{S^{n-1} \times \{t\}}) \to 0$ as $t \to \infty$.
\end{proof}

\footnotesize

\bibliographystyle{amsalpha}
\bibliography{ref}

\vspace{0.5cm}

\normalsize

\noindent{\textsc{Department of Mathematics, Heriot-Watt University, Edinburgh, UK}}

\noindent{\textit{E-mail address:} \texttt{f.fournier-facio@hw.ac.uk}} \\

\noindent{\textsc{\'Ecole Polytechnique F\'ed\'erale de Lausanne (EPFL), Switzerland}}

\noindent{\textit{E-mail address:} \texttt{nicolas.monod@epfl.ch}} \\

\noindent{\textsc{Department of Mathematics, Purdue University, US}}

\noindent{\textit{E-mail address:} \texttt{snariman@purdue.edu}} \\

\noindent{\textsc{Department of Computer and Mathematical Sciences, University of Toronto Scarborough, Canada}}

\noindent{\textit{E-mail address:} \texttt{a.kupers@utoronto.ca}}

\end{document}